\documentclass[11pt]{article}
\usepackage{custarticle}

\newcommand{\CC}{\mathbb C}

\newcommand{\NN}{\mathbb N}
\newcommand{\RR}{\mathbb R}
\newcommand{\ZZ}{\mathbb Z}
\newcommand{\Si}{\Sigma}
\newcommand{\cl}{\mathtt{cl}}
\newcommand{\heta}{\widehat{\eta}}
\newcommand{\phg}{{\mathrm{phg}}}
\newcommand{\calA}{{\mathcal A}}
\newcommand{\calB}{{\mathcal B}}
\newcommand{\calC}{{\mathcal C}}
\newcommand{\calD}{{\mathcal D}}
\newcommand{\calE}{{\mathcal E}}
\newcommand{\calF}{{\mathcal F}}

\newcommand{\calH}{{\mathcal H}}
\newcommand{\calI}{{\mathcal I}}

\newcommand{\calK}{{\mathcal K}}
\newcommand{\calL}{{\mathcal L}}

\newcommand{\calN}{{\mathcal N}}
\newcommand{\calO}{{\mathcal O}}
\newcommand{\calP}{{\mathcal P}}

\newcommand{\calT}{{\mathcal T}}
\newcommand{\calU}{{\mathcal U}}
\newcommand{\calV}{{\mathcal V}}

\newcommand{\iind}{\mathrm{Ind}\,}
\newcommand{\Kr}{\mathrm{Kr}}

\newcommand{\ie}{\mathrm{ie}}
\newcommand{\fff}{\mathrm{ff}}

\title{The index problem for $\ZZ_2$-harmonic spinors in a three-manifold branching along graphs}

\author{Andriy Haydys \\ Universit{\'e} libre de Bruxelles \\ andriy.haydys@ulb.be
\and
Rafe Mazzeo \\ Stanford University \\ rmazzeo@stanford.edu 
\and 
Ryosuke Takahashi \\ National Cheng Kung University, TAIWAN \\ tryotriple@gmail.com}

\date{}

\begin{document}
	\maketitle



\begin{abstract} 
We prove an index formula for the Dirac operator acting on two-valued spinors on a $3$-manifold $M$ which branch
along a smoothly embedded graph $\Si \subset M$, and with respect to a boundary condition along $\Si$ inspired
by an instance of this setting related to the deformation theory of $\Z_2$-harmonic spinors.  When $\Si$ is
a smooth embedded curve, this index vanishes; this was proved earlier  by one of us \cite{Takahashi18_IndexThm},
but the proof here is different and extends to the more general setting where $\Si$ also has vertices. 
We focus primarily on the Dirac operator itself, but also show how our results apply to more general twisted 
Dirac operators and to the closely related $\Z_2$ harmonic $1$-forms.
\end{abstract}

\maketitle

\section{Introduction}
Let $(M^3,g)$ be a closed, oriented, Riemannian three-manifold. Fix a spin structure on $M$; the corresponding spinor bundle $\slS$ 
is a rank $2$ Hermitian vector bundle.  Next, fix a smoothly embedded `regular' compact graph $\Si\subset M$. 
Thus $\Sigma$ is a finite collection of smoothly embedded arcs which meet in vertices; we assume that no two arcs meet tangentially at 
any vertex, and that at each vertex $v$, there is an even number $2k_v \geq 4$ of edges terminating there. We allow loops, so
both endpoints of the same edge may contribute to this valence. Now let $\cI$ be a rank $1$ real vector bundle on $M \setminus \Sigma$
endowed with a nondegenerate (i.e., nonvanishing) inner product which has monodromy $-1$ around each loop encircling any 
edge of $\Sigma$ once.  A section $\psi\in\Gamma\big ( M\setminus\Si; \slS\otimes\cI \big)$ is called a $\Z_2$ (or $\Z_2$-valued) 
spinor. Any such $\psi$ can be interpreted as a section of $\slS$ over $M\setminus \Sigma$ which is only well defined 
up to the multiplication by $\pm 1$. 

The metric $g$ induces a unique connection on both $\slS$ and $\slS \otimes \cI$, and hence determines a Dirac operator
\begin{equation}
	\label{Eq_Z2HarmSpinor}
	\Dirac := \sum_{j=1}^3 \mathtt{cl}(\bf{e}_j) \nabla^{\slS \otimes \cI} _{\bf{e}_j}\, ,
\end{equation} 
where $\mathtt{cl}(\cdot) $ indicates Clifford multiplication and $\{\bf{e}_j\}$ is a local orthonormal frame of $TM$. We say that a triple $(\Sigma, \cI, \psi)$ is a 
$\Z_2$ harmonic spinor if $\psi$ is a section of $\slS \otimes \cI$ on $M \setminus \Si$, $\Dirac \psi = 0$, 
and $|\psi| \to 0$ on $\Si$. For brevity, we usually refer to just $\psi$ alone as the $\Z_2$ harmonic spinor, with
$\Si$ and $\cI$ understood from the context. 

The hypothesis that $|\psi| \to 0$ at $\Si$ is important. Indeed, as explained below, there is an infinite dimensional 
space of $\Z_2$-twisted solutions to $\Dirac \psi = 0$ which blow up mildly along $\Si$, and in fact, there is a 
well-posed Dirichlet problem where such singular solutions are parametrized by functions on $\Si$.

$\Z_2$-harmonic spinors first appeared in Taubes' analysis of diverging sequences of flat, stable $\PSL(2;\C)$-connections, 
\cite{Taubes13_PSL2Ccmpt}, and more recently as limits of Seiberg--Witten monopoles with multiple spinors~\cite{HaydysWalpuski15_CompThm_GAFA}.
There is an obvious generalization of this notion to four-manifolds, where $\Si$ is a smooth codimension two (stratified) surface. 
These appear as limits of diverging sequences of solutions of the complex anti-self-duality equations~\cite{Taubes13_CxASD_Arx},  Seiberg--Witten equations with multiple spinors~\cite{Taubes16_SWDim4_Arx}, Kapustin--Witten
equations~\cite{Taubes18_SequencesKapustWitten_Arx}, and Vafa--Witten equation~\cite{Taubes17_VafaWitten_Arx}. 
More generally still, $\Z_2$ harmonic spinors and their non-linear analogues, called Fueter sections, also appear as limits of 
sequences of higher dimensional instantons~\cite{Walpuski17_G2InstFueter}.  There is an intriguing conjecture relating higher dimensional 
instantons and Seiberg--Witten monopoles by way of $\Z_2$ harmonic spinors and Fueter sections ~\cite{HaydysWalpuski15_CompThm_GAFA,Haydys19_G2InstantSWmonopoles}. 
All of this provides ample motivation for a deeper study of $\Z_2$ harmonic spinors.

There are numerous important questions and directions to pursue. One is to construct a robust set of examples of $\Z_2$ harmonic 
spinors. This turns out to be quite difficult, as might be expected from the overdetermined nature of the problem.  
A nonconstructive proof of existence appears in  work of Doan-Walpuski \cite{DoanWalpuski-Z2}. More recently, Taubes and 
Wu have constructed explicit homogeneous solutions in $\RR^3$ \cite{TaubesWu20_ExOfSingModels_Arx}, see also \cite{TaubesWu21}. 
The note~\cite{HMT23_NewExamples_Arx} collects a number of new examples in $3$ and $4$ dimensions, including some recently constructed
ones by He and by Sun. 

It is also important to understand the regularity of the branching set $\Si$.  Zhang proved \cite{Zhang17_Rectifiability_Arx} that 
under certain hypotheses in this three-dimensional setting, $\Si$ is rectifiable of Hausdorff dimension one. However, it is not 
unreasonable to expect that generically amongst smooth metrics, $\Si$ should be a smoothly embedded graph. (If the metric
has only finite regularity, then presumably $\Si$ also has only finite regularity.)

Another central question, and a key one in possible applications to instanton counting, is to understand whether these objects are rigid or 
exist in families. This can be phrased as follows: given 
a $\Z_2$ harmonic spinor $(\Si, \cI, \psi)$ for a given metric $g$ on $M$, do there exist $\Z_2$-harmonic spinors
$(\Si', \cI', \psi')$ associated to metrics $g'$ near $g$?  This has been addressed first by the third author \cite{Takahashi15_Z2HarmSpinors_Arx},
and more recently by Parker \cite{Parker-DefHS3M} in the case where $\Si$ is a smooth embedded curve. They prove that
there is an underlying Fredholm problem governing this deformation theory.

There are various challenges in extending this to the setting where $\Si$ is only a graph.  The present paper undertakes one 
part of this. Namely, we set up analytic machinery suitable for studying the Dirac operator on $\Z_2$-twisted spinors when
the branching set has vertices and edges, and prove various facts about this operator and a certain elliptic boundary problem 
that arises naturally for this deformation theory. The boundary condition is determined by a nowhere vanishing pair of complex-valued 
smooth functions along each edge. The motivating example is when this pair of functions arises as the leading coefficients, in a sense
to be made precise below, of a $\Z_2$-harmonic spinor. The nonvanishing hypothesis is obviously generic in general, but also
in this special case, cf.\ \cite{He-examples}.   Our main result is an index formula for $\Dirac$ acting on sections of $\slS\otimes\calI$ 
satisfying this local algebraic boundary condition. This reproves and extends the index formula in \cite{Takahashi18_IndexThm}.
When the boundary operator uses data coming from the leading coefficients of a $\Z_2$-harmonic spinor, this 
index formula is central to the analysis of the deformation 
theory \cite{Parker-DefHS3M}.   The linearized deformation operator has some additional subteties, so we do not address it
here, but shall return soon to its analysis using the results and methods here. 

This paper has two parts. In the first, we review how, when $\Si$ is a smooth closed (and possibly disconnected) curve, the analysis  
of $(\Dirac, \calT_\psi)$ lies within the framework of elliptic boundary problems for differential edge operators, as developed 
in \cite{Mazzeo91_EllTheoryOfDiffEdgeOp}, \cite{MaVe}, \cite{Usula}, \cite{Usula2}.  This leads to a new direct proof of the main result in \cite{Takahashi18_IndexThm}:
\begin{theorem} Let $\Si$ be a smooth closed curve, $\cI$ the associated twisting real line bundle and $\psi := \{(c_1^i, d_1^i)\}$ a pair 
of smooth, complex-valued functions on $\Si$ on each edge which never vanish simultaneously. These determine a boundary operator 
$\calT_\psi$ for $\Dirac$ on $L^2(M \setminus \Si; \slS \otimes \cI)$, see Section 3.8. Then $(\Dirac, \calT_\psi)$ is 
Fredholm with $\mathrm{Ind}\, (\Dirac, \calT_\psi)=0$. 
\label{theorem1}
\end{theorem}

In the second part, we study this problem when $\Si$ is a smoothly embedded graph with vertices 
$\{q_1,q_2,...,q_N\}$ connected by edges $\{e_1,e_2,...,e_J\}$. We say that $\Si$ is admissible if the valence of
each vertex $q_i$ is even; this makes it possible to define the twisting bundle $\cI$ on $M \setminus \Si$ with
holonomy $\{\pm 1\}$ around each edge. There is an extension of the edge operator theory in this setting which
can be used to analyze $\Dirac$ in this setting, with boundary conditions imposed along the edges. This `iterated edge' theory
requires an extension of the edge operator theory that was developed in \cite{ALMP} and \cite{MazzeoWitten2}.  
The `iteration' in this moniker reflects that $M \setminus \Si$ is diffeomorphic near each vertex to a cone over a punctured sphere. 
The boundary operator $\calT_\psi$ is defined on the regular part of $\Si$, and an additional decay condition is
imposed at each vertex.  Let $\rho_i(x)$ be a smooth strictly positive function on $M\setminus \{q_i\}$, with $\rho_i(x):=\dist (q_i,x)$ in 
a neighborhood of this vertex.   Fixing an $N$-tuple of real numbers $\vec \mu = (\mu_1, \ldots, \mu_N)$, define the domain
\[
\calD_{\psi, \vec{\mu}} = \Big\{u \in \rho_1^{\mu_1}\ldots \rho_N^{\mu_N} L^2(M \setminus \Si)\Big| \calT_\psi u = 0\ \text{on each edge} \Big\},
\]
Here $\calB_0 u$ denotes the pair of leading coefficients of $u$ along each edge. 
It does not seem particularly advantageous to use different weights at different vertices, so for simplicity we write $\rho = \rho_1 \ldots \rho_N$
and write $\calD_{\psi,\mu}$ to indicate that the weight function is $\rho^\mu$ for some number $\mu$.
We then consider the mapping
\begin{equation}
\Dirac: \calD_{\psi, \mu} \longrightarrow \rho^{\mu-1} L^2( M \setminus \Si).
\label{mapL0e}
\end{equation}
\begin{theorem}
Let $\Si$ be an admissible graph with twisting bundle $\cI$ on $M \setminus \Si$ and $\psi =\{(c_1^i, d_1^i)\}$ is nondegenerate, 
see Definition \ref{nondeg}. Denote by $\Lambda_i$ the spectrum of $\hat{\Dirac}_i$, the induced Dirac operator on the punctured 
sphere obtained by blowing up the vertex $p_i$. Then \eqref{mapL0e} is Fredholm so long as $\mu \notin \Lambda_i-1$.  
Moreover, if $\mu > 1/2$, then 
\[
\mathrm{ind}(\Dirac, \calD_{\psi, \mu}) = - \sum_{i=1}^N \Big(h_i + \sum_{\gamma_j^i \in (0, \mu - 1/2)\cap \Lambda_i}  M^i_{\gamma_j}\Big).
\]
Here each $h_i$ is a certain nonnegative integer determined by some underlying analytic data of the problem, and the sum is
over all indicial roots $\gamma_j^i$ such that $\gamma_j^i$ lies in the portion of the indicial root set $\Lambda_i$ at vertex $q_i$
in the interval $(0, \mu-1/2)$. The number $M^i_{\gamma_j^i})$ is the algebraic multiplicity of $\gamma_j^i$. 
\label{theorem2}
\end{theorem}

The next section provides various notation and background definitions needed throughout the paper. Section 3 provides
a careful account of the analysis of $\Dirac$ when $\Si$ is a smooth closed curve, which culminates by showing that $(\Dirac, \calT_\psi)$
is self-adjoint, which implies that the index vanishes. Section 4 describes the new features in the iterated edge calculus, and Section 5
proves Theorem \ref{theorem2}. Finally, Section 6 contains a brief discussion of the iterated edge parametrix method. 

\textsc{Acknowledgments.} AH was partially supported by the ARC grant ``Transversality and reducible solutions in the 
Seiberg--Witten theory with multiple spinors'' of the ULB. RT was partially supported by Ministry of Science and Technology 
of Taiwan under Grant no. MOST 111-2636-M-006-023.  RM wishes to express his thanks to Siqi He and Greg Parker for
many fruitful conversations.

\section{Background}
This section contains definitions and notation used throughout the paper. 

\subsection{Graphs} 
As in the introduction, $(M,g)$ denotes a fixed compact, oriented smooth Riemannian $3$-manifold.  

\begin{definition}
A compact smooth embedded graph $\Si \subset M$ is a union $\calP \cup \calE$, where $\calP = \{q_1, \ldots, q_N\}$ and
$\calE = \{e_1, \ldots, e_J\}$ are the sets of vertices and (unoriented) edges, respectively. Each $e_j$ is a smooth arc, which
for convenience are smoothly parametrized as $\gamma_j: [0,1] \mapsto e_j$ with $\gamma_j'$ nonvanishing. 
We assume that $\del' e_j := \gamma_j(0)$ and $\del'' e_j := \gamma_j(1)$ both lie in $\calP$, and allow for the possibility 
of loops, i.e., edges with $\del'e_j = \del'' e_j$.  We assume that the interior of each $e_j$ is embedded, i.e., no edge has interior 
self-intersections, and also that no two edges intersect except at their endpoints.  A graph is called {\bf regular} if, for each $q_i$, 
the `inward-pointing' tangents $\gamma_j'(0)$ and $-\gamma_j'(1)$ to each edge $e_j$ ending at $q_i$ are distinct. 
Each vertex $q_i$ has a valence $V(q_i)$, which is the number of ends of edges meeting at that vertex.  A graph is called
{\bf even} if $V(q_i) \in 2\mathbb N$ for every $i$.
\end{definition}

In this paper we restrict attention to even, regular, smooth embedded graphs, which we call {\bf admissible}. We assume that these properties hold for graphs without further comment. 

Given such a graph $\Si$, we next consider real Euclidean (i.e., metric) line bundles $\cI$ over $M\setminus \Sigma$. The Euclidean 
structure induces a natural flat connection, so up to isomorphism $\cI$ corresponds to a homomorphism $\a\colon 
\pi_1(M\setminus\Si)\to \Z_2$.  Since $\Z_2$ is abelian, $\a$ factors through a homomorphism 
$\a\colon H_1(M\setminus\Sigma)\to \Z_2$ still denoted by the same symbol. We consider only those $\a$ such
that $\a[c_j] = -1$ for all $j$, where the $c_j$ denotes a small loop in $M \setminus \Si$ encircling the edge $e_j$ once 
and encircling no other edges. Observe that restricting $\cI$ to a small sphere centered at a vertex $q_i$, we obtain a real 
line bundle over the 2-sphere punctured at $V(q_i)$ points. This latter bundle is non-trivial on each small circle in $S^2$ 
centered at any puncture. This explains the requirement that $V(q_i)$ must be even. 

\subsection{The blowup}
Many of the analytic constructions in this paper rely on methods of geometric microlocal analysis. A starting point of that 
theory is to compactify the open space $M \setminus \Si$ not simply as $M$, but rather as a manifold with corners $M_\Si$ 
obtained by blowing up $M$ around $\Si$. The asymptotic behavior of $\Z_2$ harmonic spinor fields are more 
transparent on this space. 

The construction of $M_\Si$ proceeds in two steps, corresponding to the fact that $\Si$ itself has a `two step' stratified
structure.   In the first, $M$ is blown up around the vertex set $\calP$; the result is a manifold with boundary $M_{\calP}$,
where each $q_i$ has been replaced by a $2$-sphere. The edges of $\Si$ lift to embedded curves which connect 
boundary components of this new space. The space $M_\Si$ is obtained by blowing up these lifted curves.

More carefully, $M_{\calP}$ is the union of $M \setminus \calP$ and the unit cosphere bundles $SN_{q_i} M$ of $q_i$ in $M$,
$i = 1, \ldots, N$.  This space is endowed with the obvious topology where a smooth arc in $M$ terminating at $q_i$ 
lifts to a continuous path in $M_{\calP}$ terminating at the point in $S_i$ which records its limiting direction. The smooth
structure on $M_{\calP}$ is the minimal one with respect to which lifts of smooth functions on $M$ and spherical polar coordinates 
$(\rho, \omega)$ around each $q_i$ are $\calC^\infty$. These polar coordinates provide smooth coordinate charts near the boundary
faces of $M_{\calP}$.  This blowup is also denoted $[M; \calP]$. There is a natural blowdown map $M_{\calP} \to M$. 

The edges $e_j$ lift to $M_{\calP}$ as smooth nonintersecting arcs, each still denoted $e_j$, with endpoints at the boundary 
components of $M_{\calP}$.  There is a tubular neighborhood $\calU_j$ around each $e_j$ which is identified in terms of a 
fixed trivialization of the normal bundle of $e_j$ as $[0,1] \times D_{r_0}$, where $D_{r_0}$ is the disk of radius $r_0$ in $\CC$.
The subscript $r_0$ is dropped for simplicity. We then blow up each of these edges to define the final space $M_\Si$.  
This final blowup is defined by replacing each point $z \in e_j$ with its spherical normal bundle in $M$; in this product 
identification, we replace each disk $D$ with an annulus $A$, where the inner boundary of $A$ is a copy of $S^1$ which
parametrizes the normal directions of approach to $e_j$ at $z$. The blowup restricted to $\calU_j$ is thus
$[0,1] \times A$.  The general notation which encodes these two steps is $M_\Si = [ [M; \calP] ; \calE] = [ M_{\calP}; \calE]$.
As before, there is a well defined topology and smooth structure on $M_\Si$, defined just as above. 
There are two blowdown maps: one from $M_\Si$ to $M_{\calP}$ and another from $M_\Si$ to $M$. The latter, which
is the only one we use, is denoted $\beta$. 

\begin{figure} [htb]
	\begin{center} 
		\begin{tikzpicture} 
			\draw (0,0) -- (5,0) node [above, pos=0.5] {$e_2$};
			\draw (0,0) -- (60:5) node [right, pos=0.5] {$e_1$};
			\draw (5,0) -- (60:5) node [right, pos=0.5] {$e_3$};
			
			\draw (2.5,1.43) circle (2.89);
			\filldraw (0,0) circle (0.05) node[left] {$p_3$};
			\filldraw (5,0) circle (0.05) node[right] {$p_1$};
			\filldraw (2.5,4.32) circle (0.05) node[above] {$p_2$};
			\node at (5.5,2.4) {$e_6$}; 
			\node at (-0.5,2.4) {$e_4$}; 
			\node at (2.5,-1.7) {$e_5$}; 
			
			\begin{scope} [shift={(8.2,0)}]
				
				\draw (0,0) circle (1);
				\draw (-1,0) arc (180:360:1 and 0.2);
				\draw[dashed] (1,0) arc (0:180:1 and 0.2);
				
				\draw (5,0) circle (1);
				\draw (4,0) arc (180:360:1 and 0.2);
				\draw[dashed] (6,0) arc (0:180:1 and 0.2);
				
				\draw (2.5,4.32) circle (1);
				\draw (1.5,4.32) arc (180:360:1 and 0.2);
				\draw[dashed] (3.5,4.32) arc (0:180:1 and 0.2);
				
				\draw (0.6,0.05) -- (4.4,0.05) (0.6,-0.05) -- (4.4,-0.05); 
				\draw (4.65,0.5) -- (2.8,3.7) (4.75,0.55) -- (2.9,3.75); 
				\draw (0.3,0.65) -- (2.1,3.75) (0.4,0.6) -- (2.2, 3.7); 
				
				
				\draw ([shift={(2.5,1.43)}]-17:2.84) arc[radius=2.84, start angle=-17, end angle= 77];
				\draw ([shift={(2.5,1.43)}]103:2.84) arc[radius=2.84, start angle=103, end angle= 198];
				\draw ([shift={(2.5,1.43)}]223:2.84) arc[radius=2.84, start angle=223, end angle= 317];
				
				\draw ([shift={(2.5,1.43)}]-17:2.94) arc[radius=2.94, start angle=-17, end angle= 77];
				\draw ([shift={(2.5,1.43)}]103:2.94) arc[radius=2.94, start angle=103, end angle= 198];
				\draw ([shift={(2.5,1.43)}]223:2.94) arc[radius=2.94, start angle=223, end angle= 317];
				
				\draw (0.6,0) ellipse (0.03 and 0.05);
				\draw (4.4,0) ellipse (0.03 and 0.05);
				\draw[rotate around={60:(0.35,0.625)},black](0.35,0.625) ellipse (0.03 and 0.05);
				\draw[rotate around={-60:(4.7,0.525)},black](4.7,0.525) ellipse (0.03 and 0.05);
				
				\draw[rotate around={60:(2.15,3.72)},black](2.15,3.72) ellipse (0.03 and 0.05);
				\draw[rotate around={-60:(2.847,3.73)},black](2.847,3.73) ellipse (0.03 and 0.05);
			\end{scope}
			\node at (10.7,0.3) {$S_2$};
			\node at (9.1,2.2) {$S_1$};
			\node at (12.3,2.2){$S_3$};
			\node at  (6.9,0) {$F_3$};
			\node at  (14.5,0) {$F_1$};
			\node at  (10.7,5.6) {$F_2$};
			\node at (13.8,2.4) {$S_6$}; 
			\node at (7.6,2.4) {$S_4$}; 
			\node at (10.7,-1.9) {$S_5$};  
			
		\end{tikzpicture}
	\end{center}
\caption{A graph $\Si$ and its blowup $M_\Si$. The latter is diffeomorphic to the exterior of the bounded domain shown on the right}
\end{figure}
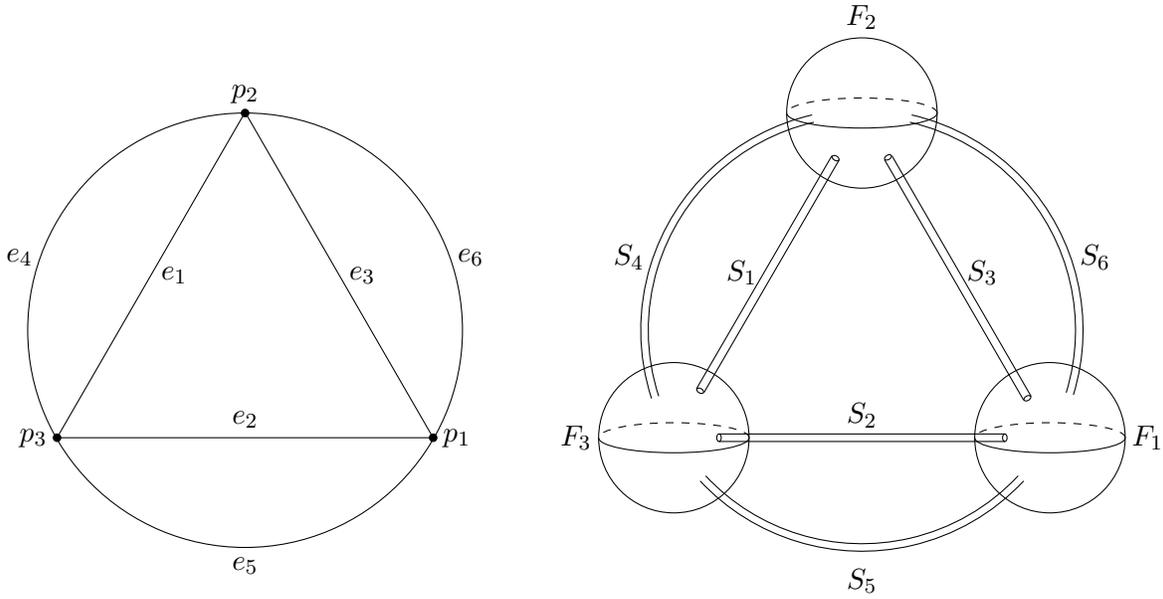

The space $M_\Si$ is a manifold with corners. It has $N + J$ boundary faces, each one covering each one of
the vertices $q_i$ or one of the edges $e_j$.  We write these faces as $S_i$ and $F_j$ respectively, with the slight
abuse of notation that each $S_i$ here is the lift of a face $S_i$ of $M_\calP$. Each
$S_i$ is a copy of $S^2$ blown up at $V(q_i)$ points, and thus a surface with $V(q_i)$ boundary components, 
each a circle. 

In terms of the fixed trivialization of the normal bundle $Ne_j$, there are cylindrical coordinates $(r,\theta, y)$ around $F_j$
in $M_\Si$. Note that $r$ is a defining function for $F_j$ in $M_\Si$, but down on $M$, the distance to $e_j$ is 
comparable to $\rho r$. 

Many constructions and results below are more naturally phrased on $M_\Si$ rather than $M$.  We can make 
more detailed and simpler statements about asymptotic behavior of harmonic spinors, for example, on $M_\Si$.
In the following we typically work near a given face $S_i$ or $F_j$, and will then often drop the index.
We always fix spherical coordinates $(\rho, \omega)$ near a given $S$, and cylindrical coordinates $(r,\theta, y)$
near $F$. Thus $y = 0$ corresponds to the intersection $S \cap F$, which is a circle, and the restrictions of $(r,\theta)$
are coordinates around this boundary of $S$.

\subsection{The twisted Dirac operator}
Any compact oriented $3$-manifold $M$ is spin, with spin structures classified by $H^1(M; \ZZ_2)$.  Having fixed both a metric $g$ 
and a spin structure on $M$, the associated spin bundle $\slS$ is a trivializable rank two complex vector bundle. The trivializations 
of this bundle are classified by $H^3(M, \ZZ) = \ZZ$, and in all of the following we fix not only the spin structure but also a trivialization 
of $\slS$.  We also fix the metric and naturally associated flat connection on $\calI$. 

These data determine the $\ZZ_2$-spin bundle $\slS\otimes \calI $ over $M\setminus\Si$, which extends smoothly to
a bundle over $M_\Si$. The induced connection is
\[
\nabla^{\slS\otimes\calI}=id_{\slS}\otimes \nabla^{\calI} +\nabla^{\slS}\otimes id_{\calI},
\]
which we typically denote simply by $\nabla$. The Dirac operator on this bundle is given by the usual formula
\[
\slD:=\sum_{i=1}^3 \mathtt{cl}({\bf e}_i)\cdot\nabla_{\bf{e}_i},
\]
where $\{ \bf{e}_1, \bf{e}_2, \bf{e}_3\}$ is any local orthonormal frame of $TM$.  Using the chosen trivialization, each
fiber $\slS_p$ is identified with a copy of $\CC^2$, so 
a section of $\slS$ becomes a $\CC^2$-valued function $u = \begin{pmatrix} u_1 \\ u_2 \end{pmatrix}$, and thus $\slD$ is a $2$-by-$2$ matrix
of differential operators.  We now write out different coordinate expressions for $\slD$ in the flat model case, where $(M,g) = \RR^3$ with 
Euclidean metric and $e$ is a straight line. 

\begin{remark}
If $E$ is any Hermitian bundle of arbitrary rank over $M$ with unitary connection $\nabla_E$, we can twist $\Dirac$ further by $E$ in the usual way
to obtain an operator $\Dirac_E$. 
The analysis below is unchanged in all essential ways, but of course the eventual index formulas may be slightly different. 
\end{remark}

Using linear coordinates $y$ along $e$ and $(x_1, x_2)$ in the normal directions, we have
\[
\Dirac = \sigma_1 \del_{x_1} + \sigma_2 \del_{x_2} + \sigma_3 \del_y, 
\] 
where
\[
\sigma_1 = \begin{pmatrix} 0 & i \\ i & 0 \end{pmatrix},\ 
\sigma_2 = \begin{pmatrix} 0 & 1 \\ -1 & 0 \end{pmatrix},\ 
\sigma_3 = \begin{pmatrix} -i & 0 \\ 0 &  i\end{pmatrix}. 
\]
In terms of the complex coordinate $z = x_1 + i x_2$ this becomes
\[
\Dirac = 2 \begin{pmatrix} 0 & i \\ 0 & 0 \end{pmatrix} \del_z + 2 \begin{pmatrix} 0 & 0 \\ i & 0 \end{pmatrix}\del_{\bar{z}} + \sigma_3 \del_y.
\]
Despite this complex notation, it is important to recall that $\Dirac$ is a {\it real} operator, and certain operations below (e.g., 
the boundary condition along the edges discussed in Section 3) are only {\it real} and do not commute with this complex structure.  
Nonetheless, complex notation is often simpler, and we use it frequently.

Now introduce cylindrical coordinates $(r,\theta, y)$, where $z = re^{i\theta}$.  Since
\[
\del_{x_1} = \cos \theta \, \del_r - \frac{\sin \theta}{r} \del_\theta\ \ \mbox{and}\ \ \del_{x_2} = \sin \theta \, \del_r + \frac{\cos \theta}{r} \del_\theta,
\]
we  obtain that
\begin{equation}
\Dirac =  \begin{pmatrix} 0 &  i e^{-i\theta} \\ i e^{i\theta} & 0 \end{pmatrix} \del_r + 
\begin{pmatrix} 0 &  e^{-i\theta} \\ - e^{i\theta} & 0 \end{pmatrix} \frac{1}{r} \del_\theta  + \sigma_3 \del_y.
\label{Dirac-cyl}
\end{equation}
The twisting by $\calI$ appears in that we consider the action of this operator on sections $u$ satisfying 
$u(r,\theta + 2\pi, y) = - u(r, \theta, y)$. 

The expression for $\Dirac$ near a vertex in this flat model takes the form
\begin{equation}
\Dirac = \sigma \Bigg((\del_\rho  + \frac{1}{\rho}) - \frac{1}{\rho} \widehat{\Dirac}\Bigg),
\label{Dirac-spherical}
\end{equation}
where $\sigma$ is Clifford multiplication with respect to $\del_\rho$ and $\widehat{\Dirac}$ is the tangential twisted
Dirac operator on $S^2$ with holonomy $-1$ around the points where this sphere intersects the outgoing rays. 

For a more general metric $g$, these expressions are the leading terms, in a precise sense, of the coordinate 
representations for $\Dirac$.  For example, using Fermi coordinates in a neighborhood of an edge $e$, 
this operator is a sum of the flat model above and an error term $E$ which is a sum of multiples of the vector fields 
$r\del_r$, $\del_\theta$, $r\del_y$ with coefficients which are smooth in $(r,\theta,y)$.  Similarly, choosing the spherical coordinates
carefully for the curved case, we can assume that the edges are straight rays emanating from the origin, and
then $\Dirac$ is the sum of the flat model \eqref{Dirac-spherical} and an error term $E$ which is a sum of a
smooth multiple of $\rho \del_\rho$ and a smoothly varying family of first order operators on $S^2$.

Since our goal in this paper is to prove an index theorem, we can eventually perform a deformation of the metric
and assume that these error terms vanish.  This simplifies the analysis somewhat.  However, the operators \eqref{Dirac-cyl} 
and \eqref{Dirac-spherical} are irremediably singular, and less standard analytic techniques are needed.  We employ
the systematic methods of the `iterated edge calculus'.  The (simple) edge calculus developed 
in \cite{Mazzeo91_EllTheoryOfDiffEdgeOp,  MaVe,  Usula,  Usula2}, provides an incisive set of tools to study 
$\Dirac$ near the edges $e_i$ away from the vertices. However, the structure of $\Si$ near each vertex is
stratified of depth two. This means (in this specific case) that near $q_i$, $M$ is diffeomorphic to a cone
where the cross-section itself is singular, here a punctured sphere.  A depth two stratified space blows up
to a manifold with corners of codimension two. The analysis of elliptic operators with iterated edge structure 
is somewhat more intricate than the simple edge case, but very similar in spirit. Material sufficient for what
we need here appears in \cite{ALMP} and \cite{MazzeoWitten2}, see also \cite{MazzeoMontcouquiol11_Stoker}.

Amongst other conclusions, this analysis leads to Fredholm (or semi-Fredholm) mapping properties of $\Dirac$ 
between various weighted spaces, and sharp regularity statements for solutions of $\Dirac u = 0$ and $\Dirac u = f$. 
These results are obtained by parametrix methods. A parametrix $G$ for $\Dirac$ is an approximate inverse (relative 
to a given function space), such that $\Dirac G - I$ and $G \Dirac - I$ are negligible remainders. The language
of blowups is very convenient for constructing this parametrix and is described in some detail in various places
below.  To make this argument more accessible, we first discuss this when $\Si$ is an embedded curve, and only
after that when $\Si$ is a graph. 

\subsection{$\ZZ_2$-harmonic spinors}
The central objects motivating this paper are the solutions of $\Dirac u = 0$ on $M \setminus \Si$ which vanish 
along $\Si$. This vanishing condition is characterized by requiring that both $u$ and $\nabla u$ lie in $L^2$. Thus define
\begin{definition}
Let $\Sigma$ be either a smooth closed curve or an admissible graph. Then
\begin{equation}
Z_0 := \big \{ u \in L^2( M_\Si; \slS \otimes \calI)\mid   \Dirac u = 0\big \}
\label{defZ0}
\end{equation}
is called the space of singular $\ZZ_2$ harmonic spinors, and its subspace
\begin{equation}
Z_1 := \big \{u \in H^1(M_\Si; \slS \otimes \calI)\mid \Dirac u = 0\big \},
\label{defZ1}
\end{equation}
is called simply the space of $\ZZ_2$ harmonic spinors or regular $\Z_2$ harmonic spinors if we wish to emphasize the distinction. 
\end{definition}

In the above definition and also in the sequel, $H^k$ denotes the Sobolev space consisting of sections $u$ with $\nabla^j u\in L^2$ for $j\le k$.  

We show below that $Z_1$ is finite dimensional, and $Z_0$ is infinite dimensional and parametrized by certain sections 
along $\Si$. 


\section{Analysis of $\Dirac$ when $\Si$ is an embedded curve}
We now turn to a detailed study of the mapping properties of $\Dirac$ and the regularity of solutions of $\Dirac u = 0$
in the simpler situation where $\Si$ is a smooth closed curve.  Some parts of this appeared in \cite{Takahashi15_Z2HarmSpinors_Arx}
and \cite{Parker-DefHS3M}, but as noted earlier, will be treated as an instance of the edge operator theory. The material here is drawn from \cite{Mazzeo91_EllTheoryOfDiffEdgeOp} adapted to this particular setting.

\subsection{Background for the edge calculus}

We first explain some basic structures underlying the edge calculus.

\medskip

The starting point is that the operator $\Dirac$, as expressed in \eqref{Dirac-cyl}, is an {\bf incomplete edge operator}.  We explain
this terminology. When $\Si$ is a smooth embedded curve, $M_\Si$ is a manifold with boundary $F = \del M_\Si$, where $F$ is the
total space of an $S^1$ bundle over $\Si$. We define the class $\calV_e$ of smooth {\bf edge vector fields} on $M_\Si$ as consisting of
those smooth vector fields which are unconstrained in the interior and which lie tangent to these $S^1$ fibers at $F$.   In local
cylindrical coordinates, any $V \in \calV_e$ is a linear combination with smooth coefficients of the coordinate vector fields 
$r\del_r$, $r\del_y$ and $\del_\theta$. Note that $\calV_e$ is Lie subalgebra of the space of all smooth vector fields. A differential 
operator on $M_\Si$ (or on $M \setminus \Si$) is called a differential edge operator if it is locally the sum of products of elements of 
$\calV_e$  with smooth (possibly endomorphism-valued) coefficients. The space of differential edge operators of order $k$ is denoted 
$\mathrm{Diff}_e^k(M_\Si)$.  Our operator $\Dirac$ is almost of this form, and indeed $r \Dirac \in \mathrm{Diff}_e^1(M_\Si)$ (with
values in the endomorphisms of the spinor bundle).  The modifier `incomplete' signifies that $\Dirac$ is $1/r$ times an edge operator.

\medskip

The smooth volume form $dV_g$ on $M$ lifts to a smooth nonvanishing
multiple of $r dr d\theta dt$ on $M_\Si$.  Considering first just scalar functions, the space $L^2(M_\Si; dV_g)$ leads
naturally to the edge Sobolev spaces of positive integer order:
\begin{definition}
For every $k\in \mathbb{N}$, we define the {\bf edge Sobolev space} to be
\[
H^k_e(M_\Si; dV_g) = \big \{ u \in L^2\mid \   V_1 \ldots V_j u \in L^2\ \mbox{for any}\ V_i \in \calV_e\ \mbox{and}\ j \leq k\big \}.
\]
\end{definition}
From these we can define edge Sobolev spaces of noninteger order by interpolation, and spaces of negative order
by duality.   We also define the weighted edge Sobolev spaces
\[
r^\mu H^s_e(M_\Si; dV_g) =\big  \{ u = r^\mu v\mid \  v \in H^s_e\big\}
\]
for any $\mu, s \in \RR$.   Notice that (in some region $\{ r < r_0\}$)
\begin{equation}
r^\lambda \in r^{\mu} L^2(r dr d\theta dy) \quad\Leftrightarrow\quad r^{\lambda - \mu} \in L^2(rdr d\theta dy)\quad \Leftrightarrow \quad\lambda > \mu - 1.
\label{indrindw}
\end{equation}
In other words the `threshold' weight $\mu$ for which $r^\lambda \in r^{\mu} L^2$ is $\mu = \lambda + 1$. 

These function spaces can be defined immediately also for sections of any bundle on $M_\Si$, and in particular for $\slS \otimes \calI$. 

The following is an immediate consequence of the definitions: 
\begin{proposition}
For any $\mu, s \in \RR$, the mapping
 \begin{equation}
\Dirac:  r^{\mu} H^{s+1}_e(M_\Si ,  \slS  \otimes \calI) \longrightarrow r^{\mu-1} H^s_e(M_\Si ,  \slS  \otimes \calI)
 \label{Dmu}
\end{equation}
is bounded. 
\end{proposition}
In particular, when $\mu=1$, $s=0$, the mapping
\begin{align*}
\slD: r H^{1}_e(M_\Si ,  \slS  \otimes \calI) \longrightarrow  L^2(M_\Si ,  \slS  \otimes \calI).
\end{align*}
is actually defined on the traditional Sobolev spaces by the following lemma.
\begin{lemma}\label{rH=L21}
When $\Sigma$ is a smooth closed curve, $rH^1_e=H^1$.
\end{lemma}
\begin{proof}
By the definition of $rH^1_e$, it is a completion of smooth sections equipped with the norm
\begin{align*}
\|u\|_{rH^1_e}=\Big(\int_{M\setminus\Sigma}\Big(\frac{|u|}{r}\Big)^2+|\nabla u|^2\Big)^{\frac{1}{2}}.
\end{align*}
The argument of Lemma 2.6 in \cite{Takahashi15_Z2HarmSpinors_Arx} shows that this norm is equivalent to $H^1$-norm.
\end{proof}

In the next several subsections we introduce tools and results which allow us first to describe
the precise regularity of solutions to $\Dirac u = 0$ near $\Si$, where $u$ is assumed to lie in some $r^\mu L^2$,
and then to state the key mapping properties of $\Dirac$. We then formulate a Fredholm boundary problem 
for this operator which is intimately related with the infinitesimal deformations of $\Z_2$-harmonic spinors. 

\subsection{Indicial roots}
The first main goal is to identify the values of $\mu$ for which \eqref{Dmu} has a chance of being `well-behaved'.  It turns out that 
there is no value of $\mu$ for which this mapping is Fredholm.  This is not due to an injudicious choice
of function space, but instead reflects important features of this problem.   As we explain now, this mapping
fails to have closed range only for a discrete set of values of $\mu$, but for all other values of $\mu$ it has closed range,
but either its nullspace or cokernel are infinite dimensional. 

One component in understanding these phenomena is the existence of good approximate solutions of the form
$u(r,\theta,y) = r^\lambda \tilde{u}(\theta,y)$ for certain values of $\lambda$.    Notice that if $\tilde{u}$ is smooth,
then $\Dirac ( r^\lambda \tilde{u}) = r^{\lambda-1} f$ for some function (or section) $f$ which is smooth for $r \geq 0$. 
The value $\lambda$ is called an indicial root of $\Dirac$ if there exists some smooth $\tilde{u}(\theta,y)$  such that
\[
\Dirac (r^\lambda \phi) = \calO(r^{\lambda});
\]
in other words, there is a cancellation of the leading order term on the right.  We tacitly assume that $\lambda$ is real,
but this definition still makes good sense and holds even when $\lambda$ is complex.  

To calculate these indicial roots, first notice that since $\del_y$ appears in $\Dirac$ with a smooth nonsingular 
coefficient, this term plays no role in the calculation of $\lambda$.  In fact, the $y$-dependence of $u$ plays
no role either. The definition is also unaffected if the coefficient function $\tilde{u}$ depends smoothly on $r$, 
since we can just replace it by its leading Taylor coefficient at $r=0$.  Hence we may as well replace $\Dirac$ 
by its flat model, i.e., the operator when $M = \RR^3$ and $\Si$ is a straight line.  With that, it is completely
natural to write sections of $\slS \otimes \calI$ as column vectors with two components, and incorporate the $\Z_2$ 
twist by imposing the antiperiodicity condition $\tilde{u}(\theta+2\pi, y) = -\tilde{u}(\theta,y)$.  

Using the antiperiodicity, the eigenfunctions for the tangential operator $(1/i) \del_\theta$ are $e^{i(k+\frac12)\theta}$, $k \in \ZZ$.  
Reducing to each eigenspace, it suffices finally to seek solutions to the model equation $\Dirac u = 0$ where
\[
u( r, \theta) = \begin{pmatrix} a \\ b \end{pmatrix}  e^{i (k + \frac12)\theta} r^\lambda
\]
for constants $a, b \in \CC$.  We compute then that
\[
\begin{split}
& \Dirac \begin{pmatrix} 0 \\ 1 \end{pmatrix} e^{i(k+\frac12)\theta} r^\lambda = 
\begin{pmatrix}  1 \\ 0 \end{pmatrix} ( i e^{i(k-\frac12)\theta} (\lambda + (k+\tfrac12)) r^{\lambda-1}, \\ 
& \Dirac \begin{pmatrix} 1 \\ 0 \end{pmatrix} e^{i(k+\frac12)\theta} r^\lambda = 
\begin{pmatrix}  0 \\ 1 \end{pmatrix} ( i e^{i(k+\frac32)\theta} (\lambda - (k+\tfrac12)) r^{\lambda-1}. 
\end{split}
\]
and from this we may easily deduce the values of the indicial roots: 

\begin{proposition}
The indicial roots of $\Dirac$ consist of the values $\lambda_k  := k + \tfrac12$, $k \in \ZZ$, with coefficient functions $e^{i(k+\frac12)\theta}$.
\end{proposition}
With \eqref{indrindw} in mind, define the {\bf indicial  weights} of $\Dirac$ to be the set of threshold weights corresponding to these indicial 
values namely 
\[
\mu_k = \lambda_k + 1 = k + 3/2, \quad k \in \ZZ.
\]

As we explain below, the indicial root $-1/2$ (and corresponding indicial weight $+1/2$) play a particularly special role in the analysis of
this operator. 

\subsection{The normal operator}
The indicial roots of $\Dirac$ depend on only the simplest leading order behavior of this operator as $r \to 0$. There is a more
refined model for $\Dirac$ at $r = 0$ which carries more determinative information.  This is the so-called normal operator $N(\Dirac)$. 
There is an invariant definition of this normal operator which proceeds by taking a limit of pullbacks by dilation in the $(r,y)$
variables around any given point $(0, y_0)$. However, it can also be obtained by taking the leading order terms in the 
Taylor expansions of the coefficients of the vector fields $r\del_r$, $r\del_y$ and $\del_\theta$ in the local coordinate 
expression for $r \Dirac$.  It should be fairly obvious that this is just what we have been calling the flat model above, namely
\begin{equation}
N(\Dirac) =  \begin{pmatrix} 0 &  i e^{-i\theta} \\ i e^{i\theta} & 0 \end{pmatrix} \del_s + 
\begin{pmatrix} 0 &  e^{-i\theta} \\ - e^{i\theta} & 0 \end{pmatrix} \frac{1}{s} \del_\theta + 
\begin{pmatrix} -i  & 0 \\ 0 & i\end{pmatrix} \, \del_w. 
\label{no}
\end{equation}
Note that we have changed the names of the variables $r$ to $s$ and $y$ to $w$ to
emphasize that $N(\Dirac)$ acts as an operator on the entire space $\RR^+_s \times S^1_\theta \times \RR_w$ and not
just in a local neighborhood near $\{(0, \theta, y_0)\}$.

The normal operator $N(\Dirac)$ is a `boundary symbol': its global mapping properties 
determine the mapping (and other regularity) properties of $\Dirac$ itself.  Note first that 
\begin{equation}
N(\Dirac): s^\mu H^1_e( \RR^+ \times \RR \times S^1) \to s^{\mu} L^2(\RR^+\times \RR \times S^1)
\label{normalmapping}
\end{equation}
for any $\mu$. The last term involving $\del_w$ does not have a decaying coefficient, so $N(\Dirac)$ does not map into $s^{\mu-1}L^2$. 
Notice that here both limiting behaviors as $s\to 0$ and $s\to +\infty$ are important.

We claim that \eqref{normalmapping} is semi-Fredholm whenever $\mu \not\in \ZZ + \tfrac12$; in other words, so long as $\mu$ 
is nonindicial, this mapping has closed range and either finite dimensional nullspace or finite dimensional cokernel.  In fact,
since $N(\Dirac)$ is translation invariant in $w$ and has homogeneity $-1$ with respect to dilations in $(s,w)$, 
if $u$ is a nontrivial element of the nullspace for a given $\mu$, then its translates and dilates span an infinite dimensional 
space and are all in the nullspace.  This means that even a $1$-dimensional nullspace in $s^\mu L^2$ implies that the nullspace
is infinite dimensional. There is a similar statement for the cokernel.  Now, it follows from calculations below that if
$u \in s^\mu L^2$ lies in the nullspace of $N(\Dirac)$ for some $\mu > 1/2$, then $u \equiv 0$. We then deduce from 
the formal symmetry of $N(\Dirac)$ on $L^2$ and a duality argument that $N(\Dirac)$ is surjective when $\mu < 1/2$ (and 
$\mu$ is nonindicial).

To prove this injectivity statement, we use the same symmetries to reduce $N(\Dirac)$ to a simpler operator.  Indeed, taking
Fourier transform in $w$ gives
\[
\widehat{N(\Dirac)} =  \begin{pmatrix} 0 &  i e^{-i\theta} \\ i e^{i\theta} & 0 \end{pmatrix} \del_s + 
\begin{pmatrix} 0 &  e^{-i\theta} \\ - e^{i\theta} & 0 \end{pmatrix} \frac{1}{s} \del_\theta  + i \cl (\eta) \otimes \begin{pmatrix} -i  & 0 \\ 0 & i\end{pmatrix} ;
\]
this depends parametrically on the Fourier dual variable $\eta$.  Setting $t  = s |\eta|$, 
$\heta = \eta/|\eta| \in \{\pm 1\}$, this reduces further to $|\eta|^{-1}B(\Dirac)$, where
\[
B(\Dirac) =  
\begin{pmatrix} 0 &  i e^{-i\theta} \\ i e^{i\theta} & 0 \end{pmatrix} \del_t + 
\begin{pmatrix} 0 &  e^{-i\theta} \\ - e^{i\theta} & 0 \end{pmatrix} \frac{1}{t} \del_\theta + \cl(\heta) \otimes \begin{pmatrix} 1  & 0 \\ 0 & -1\end{pmatrix}.
\]
This is called the {\em model Bessel} operator for $\Dirac$.

The Fourier transform and rescaling above are reversible operations, so the assertion about the mapping properties
of $N(\Dirac)$ can be deduced immediately from the following result.
\begin{proposition}\label{BDmapprop}
The model Bessel operator 
\begin{equation}
B(\Dirac): t^{\mu} H^1_b(\RR^+ \times S^1; t dt d\theta) \longrightarrow t^{\mu} L^2(\RR^+ \times S^1; t dt d\theta)
\label{Bmap}
\end{equation}
is Fredholm if and only if $\mu \not\in \ZZ + \frac12$.  It is injective, with nontrivial cokernel, if and only if $\mu > 1/2$, and 
surjective with nontrivial nullspace,  if and only if $\mu < 1/2$ (and $\mu \not\in \ZZ + \frac12$). 
\end{proposition}
\begin{remark}
The space $H^1_b$ consists of those functions or sections $v \in L^2$ such that $(r\del_r)v$ and $\del_\theta v$ are also in $L^2$;
the space $t^\mu H^1_b$ consists of functions $v = t^\mu w$ where  $w \in H^1_b$.
\label{defH1b}
\end{remark}
We indicate a proof in the next subsection. 

\subsection{The model Bessel operator $B(\Dirac)$}
We begin by showing the relevance of the condition that $\mu$ is nonindicial. 
\begin{proposition}
If $\mu = \mu_k = k + \tfrac32$ for some $k$, then $B(\Dirac): t^\mu H^1_b \to t^{\mu} L^2$ does not have closed range. 
 \end{proposition}
\begin{remark}
The proof below applies equally well to $N(\Dirac)$ and $\Dirac$.
\end{remark}
\begin{proof}
If $\mu = \mu_k$, consider a sequence of sections $\psi_\ell(t,\theta) = \chi_\ell(t)  t^{k+1/2} e^{i(k+1/2)\theta}$, where
$\chi_\ell(t)$ is a smooth cutoff function with support in $[2^{-\ell-1}, 2^{-\ell+1}]$, $\ell \in \NN$.  For example,
we can take $\chi_\ell(t) = \tilde{\chi}( \log t + \ell)$, where $\tilde{\chi}(\tau) \in \calC^\infty$ has support in
$|\tau| \leq 1$ and equals $1$ for $|\tau| \leq 1/2$. It is straightforward to check that this is a Weyl sequence for 
$B(\Dirac)$, i.e., $|| \psi_\ell||_{t^{\mu_k} L^2} \cong 1$ while $||\Dirac \psi_\ell||_{t^{\mu_k} L^2} \to 0$. Hence
the range is not closed.
\end{proof}

Now recall how duality works in this setting.   Using the standard Riemannian volume measure, $\int (\Dirac u) v = \int u (\Dirac v)$,
and similarly, with respect to $t dt d\theta$, $\int B(\Dirac) u v = \int u B(\Dirac) v$.  Putting aside the issue of regularity for the moment, 
the dual of the mapping $B(\Dirac): t^{\mu} L^2 \to t^{\mu-1} L^2$ with respect to the (unweighted!) $L^2$ pairing is 
$B(\Dirac): t^{1-\mu} L^2 \to t^{-\mu} L^2$.  This duality is symmetric about $\mu = 1/2$, the indicial weight associated to the 
indicial root $-1/2$. Thus if $\calK(\mu)$ and $\calC(\mu)$ denote the dimensions of the kernel and cokernel of \eqref{Bmap}, 
and $\iind(\mu)$ its index, then 
\[
\dim \calK(\mu) = \dim \calC(1-\mu), \qquad  \iind(1-\mu) = - \iind(\mu).
\]
This same duality holds for $N(\Dirac)$ and for $\Dirac$ itself, but unlike for $B(\Dirac)$, the indices of these operators are $\pm \infty$. 

We now determine $\calK(\mu)$. Writing out the system $B(\Dirac) \psi = 0$ gives
\[
\begin{split}
(t\del_t + i \del_\theta) \psi_1  & = -\, i \, \heta \, t \, e^{-i\theta} \psi_2 \\
(t\del_t - i \del_\theta) \psi_2 & = \, i \, \heta \, t\, e^{i\theta} \psi_1;
\end{split}
\]
this is more tractable if we set $\tilde{\psi}_1 = e^{i\theta} \psi_1$, which yields
\[
\begin{split}
(t\del_t + i\del_\theta + 1) \tilde{\psi}_1 & = - t \, i \heta \psi_2 \\
(t\del_t - i\del_\theta) \psi_2 & =  t \, i \heta \tilde{\psi}_1. 
\end{split}
\]

Eliminating first $\psi_2$ and then $\tilde{\psi}_1$ produces the two uncoupled equations
\[
\begin{split}
  \left( (t\del_t )^2   - (i\del_\theta+1)^2  - t^2\right) \tilde{\psi}_1 & = 0 \\
  \left( (t\del_t)^2 + \del_\theta^2 - t^2\right) \psi_2  & = 0.
  \end{split}
\]
Writing $\tilde{\psi}_1 = A(t)  e^{i(k+\tfrac12)\theta}$ and $\psi_2 = B(t) e^{i(k + \tfrac12)}$, we see that 
\[
  t^2 A'' + t A' - ( \nu_{k-1}^2 + t^2) A = 0, \quad t^2 B'' + t B' - (\nu_k^2 + t^2) B = 0,
\]
where
\[
\nu_k = |k + \tfrac12|,\ \ k \in \ZZ \ \ \  \mbox{(so $\nu_{-1} = \nu_0 = \tfrac12$, $\nu_{-2} = \nu_1 = \tfrac32$, etc.).}
\]
The solutions are Bessel functions of imaginary order: 
\[
\begin{split}
 A(t) =c_1  I_{\nu_{k-1}}(t) +  c_2 K_{\nu_{k-1}}(t), \ B(t) = c_1'  I_{\nu_k} (t) +  c_2' K_{\nu_k}(t),\ \ k \in \ZZ.
\end{split}
\]
Referring to \cite{Lebedev}, up to multiplicative constants, since $\nu > 0$ and $\nu \not\in \ZZ$, 
  \[
I_\nu(t) \sim t^\nu\ \mbox{as}\ t \searrow 0, \ \ \ I_\nu(t) \sim t^{-1/2} e^t\ \ t \nearrow \infty,
\]
while
\[
K_\nu(t) \sim t^{-\nu}\ \mbox{as}\ t \searrow 0, \ \ \ K_\nu(t) \sim t^{-1/2} e^{-t}\ \ t \nearrow \infty.
\]

After a further calculation, we obtain finally that solutions to the original coupled system $\Dirac \psi = 0$ are necessarily
superpositions of solutions of the form
\[
a_k \begin{pmatrix}  I_{\nu_{k-1}}(t) e^{i(k-\tfrac12)\theta}  \\ c_k(\heta)  I_{\nu_k}(t) e^{i(k + \tfrac12)\theta} \end{pmatrix} + 
b_k  \begin{pmatrix}  K_{\nu_{k-1}}(t) e^{i(k-\tfrac12)\theta}  \\ d_k(\heta) K_{\nu_k}(t) e^{i(k + \tfrac12)\theta} \end{pmatrix},
\]
for some coefficients $a_k, b_k$ and where $c_k(\heta)$ and $d_k(\heta)$ are determined by the equation.   To
work in the realm of tempered distributions, we discard solutions growing exponentially in $t$, so we must set $a_k = 0$ 
and then may as well take $b_k = 1$. We may then also calculate the values $d_k(\heta)$.  This information is needed 
only for $k = 0$, and, since $K_{1/2}(t)$ is a constant multiple of $t^{-1/2} e^{-t}$,  we arrive at the particular solution
\begin{equation}
\begin{pmatrix}  e^{-i\theta/2}  \\ - i \heta \, e^{i\theta/2} \end{pmatrix} t^{-\frac12} e^{-t}.
\label{psoln}
\end{equation}
We have singled out this one solution because we shall work in $r^{\mu} L^2$, $-1/2 < \mu < 1/2$, and this excludes all solutions 
except those with $\nu = \pm \tfrac12$. This leaves only $k=0$, and the specific solution \eqref{psoln}.

The duality considerations and calculations above now yield the following result. 
\begin{proposition}
If $-1/2 < \mu < 1/2$, then $B(\Dirac): t^\mu H^1_b \to t^{\mu} L^2$ is surjective, with one dimensional nullspace
spanned by the explicit solution \eqref{psoln}.  If $1/2 < \mu < 3/2$, then $B(\Dirac): t^\mu L^2 \to t^{\mu} L^2$ is 
injective, with $1$-dimensional cokernel.   There is no value of the weight parameter for which $B(\Dirac)$ is an isomorphism.
\end{proposition}
\begin{remark}
The last assertion shows that neither $N(\Dirac)$ nor $\Dirac$ are Fredholm on any weighted $L^2$ space. Indeed, a central
result in \cite{Mazzeo91_EllTheoryOfDiffEdgeOp} states that $\Dirac: r^\mu H^1_e \to r^{\mu-1}L^2$ is Fredholm if and only if
$B(\Dirac): t^\mu H^1_b \to t^{\mu}L^2$ is an isomorphism, and we have just explained that this can never be the case. 
The difference in the weight of the range spaces for $B(\Dirac)$ and $\Dirac$ is because, for $B(\Dirac)$, the term of
order $0$ which involves $\hat{\eta}$ does not decay; in $\Dirac$, this term corresponds to a tangential derivative,
and its effect on the weight is taken into account in the definition of the edge Sobolev spaces. 
\end{remark}

\subsection{Mapping properties of $\Dirac$} 
Passing from $\Dirac$ to its normal operator and then its model Bessel operator is analogous to considering
the principal symbol of a differential operator. The (left or right) inversion of the Bessel operator can be
parlayed, by rescaling and inverse Fourier transform, to the similar left or right inversion of the normal operator.
The model pseudoinverse obtained this way can then be successively corrected to a left or right parametrix
for $\Dirac$ itself.   Note that we have been careful to state this in terms of one-sided (approximate) inverses, which is 
all that is possible for this problem.  The steps where this successive correction take place are the analogues,
in the pseudodifferential edge calculus, of the standard elliptic parametrix construction for nondegenerate
elliptic operators, with the important proviso that the construction is done at the level of the Schwartz kernel
rather than the symbol.  This is described carefully elsewhere, but we shall give a brief review of how this
all goes at the end of this section. 

However, let us describe the result of this whole process. Fix any $\mu \not\in \ZZ + \frac12$.  If $\mu > 1/2$, 
then there exists a left parametrix $G$. This is a pseudodifferential operator of order $-1$ with a particular 
singular structure near $\Si$ which satisfies
\[
G \circ \Dirac = I - R_1,
\]
where $R_1$ is a {\it residual} operator. This remainder term maps $r^\mu H^s_e$ into $r^{\mu'}H^{s'}_e$
for some $\mu' > \mu$ and for all $s' \in \RR$. (Later we shall state its regularity gain even more accurately.)
Since $r^{\mu'}H^{s'}_e \hookrightarrow r^{\mu}H^{s}_e$ is compact, we conclude that $\Dirac$ has
at most a finite dimensional nullspace in $r^\mu H^s_e$ for any $s$.  This parametrix can also be used
to prove that the range of $\Dirac $ is closed. Indeed, if $\Dirac u_j = f_j$ and $f_j$ converges
in $r^{\mu-1}L^2$, then using that $u_j = G f_j + R_1 u_j$, and adjusting $u_j$ by an element of the nullspace
so that $u_j$ is orthogonal to this nullspace, we conclude that $u_j$ converges in $r^\mu H^1_e$. 

Similarly, if $\mu < 1/2$, $\mu \not\in \ZZ + \frac12$, then there exists an edge pseudodifferential
operator $G'$ of order $-1$ such that $\Dirac \circ G' = I - R_1'$, where $R_1'$ is residual in the
same sense as above. This equality of operators implies that the range of $\Dirac$ is closed and
of finite codimension.

The operators $G$ and $G'$ depend on the choice of weight parameter, of course. It turns out that one
can choose these operators to be constant in each interval $(k -1/2, k+1/2)$, $k \in \ZZ$.  The rough
behavior is that as $\mu$ increases, the nullspace of $\Dirac$ is smaller and smaller, and eventually
becomes trivial for large enough $\mu$, but the cokernel jumps up (but remains infinite dimensional)
as $\mu$ crosses each indicial weight.  Similarly, as $\mu$ decreases, $\Dirac$ eventually becomes
surjective, but its nullspace similarly jumps up (and is infinite dimensional) as $\mu$ crosses the
negative values in $\ZZ + \frac12$. 

We shall henceforth primarily be interested in the two cases where $\mu \in (-1/2, 1/2)$, so that
$\Dirac$ is surjective modulo a finite dimensional error, and $\mu \in (1/2, 3/2)$, where $\Dirac$ 
is injective modulo a finite dimensional error.  In particular, we may as well study just the
two cases $\mu = 0$ and $\mu = 1$.

\begin{proposition}
The space of $\Z_2$-harmonic spinors can also be written as
\[
Z_1 = \{ u \in r L^2(M_\Si: \slS \otimes \calI)|  \Dirac u = 0 \}.
\]
\label{defZ}
\end{proposition}
\begin{proof}
A simple argument using the parametrix  for $\Dirac$ on $r H^1_e$ can be used to show that if $u \in Z_1$, then $u \in r H^1_e$,
and in fact $u \in r H^\ell_e$ for any $\ell \geq 0$. Now, by using Lemma \ref{rH=L21}, we obtain this result.
\end{proof}
Hence, by the results above, $Z_1$ is finite dimensional, while $Z_0$ is infinite dimensional.  

\begin{theorem}\label{Dmapprop}
The mapping
\[
\Dirac:  r H^1_e( M_\Si; \slS  \otimes \calI) \longrightarrow L^2( M_\Si; \slS \otimes \calI)
\]
has finite dimensional nullspace $Z_1$ and closed range which can be identified with the orthogonal
complement of $Z_0$.  On the other hand, the nullspace of the mapping
\[
\Dirac:   H^1_e( M_\Si; \slS  \otimes \calI) \longrightarrow r^{-1} L^2( M_\Si; \slS \otimes \calI)
\]
is $Z_0$, and its range is a closed subspace of finite codimension, characterized as the annihilator of the finite
dimensional space $Z_1$ with respect to the unweighted $L^2$ pairing.
\label{mainmapD}
\end{theorem}
\begin{proof}
Proposition~\ref{defZ} identifies the two nullspaces.  Suppose $f \in L^2$ and $f \perp \Dirac u$ for all
$u \in r H^1_e$.  We may test against only those $u$ which are smooth and compactly supported, in
which case it is trivial that $0 = \langle f, \Dirac u\rangle = \langle \Dirac f, u \rangle$ for all such $u$,
so $f \in Z_0$.  Similarly, the same calculation shows that $Z_0$ is orthogonal to the range. 

The corresponding statement for the other map is proved in the same way.
\end{proof} 

It is exceptional for $Z_1$ to contain anything other than $0$, whereas $Z_0$ is always infinite dimensional for any $\Si$. 

\subsection{Asymptotics of solutions}
We next turn to the regularity of elements of $Z_1$ and $Z_0$. The starting point is a pair of results from the general 
edge operator theory,  adapted to the present situation.   Both of the results below are local, i.e., they hold for 
solutions defined only in a neighborhood $\calU$ around a point $y_0 \in \Si$, but not necessarily globally on 
$M \setminus \Si$.  (In particular, these asymptotic results hold near points in the interior of any edge when $\Si$ is a graph.)
The first states that any solution of $\Dirac u = 0$ which does not blow up as fast as the `critical' indicial root $-1/2$ 
is necessarily polyhomogeneous. The second asserts the existence of a `weak' expansion for any solution $u \in L^2$. These
solutions might actually blow up more quickly than $r^{-1/2}$, but when `averaged' in the $\Si$ direction, blow up at precisely
the rate $r^{-1/2}$. This generalizes the well-known behavior of bounded harmonic functions in a half-space, and in any case
leads to the optimal regularity statement in this setting.

In the following, we shall use the term polyhomogeity freely.  This is a regularity condition slightly more general
than smoothness, which is particularly useful for describing solutions of singular elliptic PDE.  
\begin{definition} If $M_\Si$ is a manifold with boundary, then a function $u$ is called polyhomogeneous on $M_\Si$ at the boundary if
\[
u \sim \sum_{j=0}^\infty \sum_{\ell = 0}^{K_j}  a_{j\ell}(\theta, y) r^{\gamma_j} (\log r)^\ell,
\]
where $\{\gamma_j\}$ is a discrete sequence of complex numbers with real parts converging to infinity.   The meaning
of this asymptotic expansion is the classical one, namely that the difference between $u$ and any finite
sum of this series decays faster than the next term in the series, and this asymptotic equivalence is also true if
we differentiate both sides any number of times with respect to any of these variables. In particular, the coefficients
$a_{j\ell}$ are all $\calC^\infty$. 
\end{definition}
The special case where $\gamma_j = j$ and every $K_j = 0$ corresponds to $u$ being smooth for $r \geq 0$.  
In the only cases considered here, the $\gamma_j$ lie in subset of either $\ZZ$ or else $\ZZ + \frac12$ which is bounded below. 

\begin{proposition}
Let $\calU = \calV \times B_\epsilon $ be a neighborhood in $M$ near a point $y_0 \in \Si$, where $\calV$ is an open set 
in $\Si$ and $B_\epsilon$ is a small ball in $\RR^2$, identified with the fibers of $N\Si$. If $u \in  r^\mu 
L^2( \calU \setminus \Sigma; \slS \otimes \calI)$ satisfies $\Dirac u = 0$ and $\mu > 1/2$, then $u$ is
polyhomogeneous with an asymptotic expansion
\[
u(r, \theta, y) \sim \sum_{k \geq k_0} \sum_{\ell =-k + k_0}^{k-k_0}  \begin{pmatrix} a_{k,\ell}(y) e^{-i \theta/2 } \\ b_{k,\ell} (y) 
e^{i \theta/2} \end{pmatrix} e^{i\ell \theta} \, r^{k -\frac12} (\log r)^\ell, 
\]
where every coefficient $a_{k,\ell}(y)$ and $b_{k,\ell}(y)$ lies in $\calC^\infty(\calV)$ and $k_0$ is the unique integer such that $|\mu - k_0| < 1$. 
\label{sexp}
\end{proposition}
To interpret this expansion more easily, we note that in the `flat' case, where $M = \RR^3$ and $\Sigma = \RR$, 
the sum in $\ell$ is absent and only the monomial $e^{ik\theta}$ appears as a coefficient of $r^{k-\frac12}$.  The extra logarithmic factors are needed
to ensure that $\psi$ is a `formal' solution (i.e., a solution in the sense of series expansions) in the presence of correction
terms arising from the fact that $M$ and $\Sigma$ are curved. 

We are particularly interested in the case where $1/2 < \mu < 3/2$, so $k_0 = 1$, since this is the case where there exists a solution to $\Dirac u = 0$
with $u \sim r^{1/2}$. 

\begin{proposition}
Let $\calU = \calV \times B_\epsilon$ be a neighborhood in $M$ around a point of $\Sigma$, as above. 
Suppose that $u \in L^2(\calU \setminus \Si; \slS \otimes \calI)$ satisfies $\Dirac u = 0$. Then
\[
u(r, \theta, y) \sim \sum_{k = 0}^\infty \sum_{\ell = -k}^k \begin{pmatrix} a_{k,\ell}(y) e^{-i \theta/2} \\ b_{k,\ell}(y) e^{i\theta/2} \end{pmatrix} 
e^{i\ell \theta} r^{k -\frac12} (\log r)^\ell,
\]
where $a_{k,0}, b_{k,0} \in H^{-1/2 - k}(\calV)$ for all $k$ and $a_{k,\ell}, b_{k,\ell} \in H^{-1/2-k-\epsilon}$ for any $\epsilon > 0$ if $\ell \neq 0$. 
In particular, the `leading' coefficients $a_{0,0}(y), b_{0,0}(y)$ lie in $H^{-1/2}(\calV)$.   We refer to  \cite{Mazzeo91_EllTheoryOfDiffEdgeOp} for a detailed interpretation of this type of weak expansion.
\label{wexp}
\end{proposition}
\begin{remark}
Since these results are local, they hold for solutions $u$ of $\Dirac u = 0$ branched along a graph locally near interior points 
of edges of that graph.
\end{remark}
\begin{remark}
In the following, we do not need to keep track of the terms with logarithmic coefficients too carefully; we have recorded their precise structure
here for the sake of accuracy. However, we shall almost exclusively be considering the leading coefficients of these expansions. For that reason,
we shall more simply denote the leading coefficients by $a_0, b_0$ for this second theorem, and by $a_1, b_1$ in the first when $u
\in r^\mu L^2$ for $\mu \in (1/2, 3/2)$. 
\end{remark}
\begin{proof}[Sketch of proofs] These results rely on the existence and pointwise properties of left parametrices for 
$\Dirac$ as elements of the edge calculus.  First suppose that the weight $\mu$ lies in $(1/2, 3/2)$.  In this case, 
$B(\Dirac): t^{\mu}H^1_b \to t^{\mu-1}L^2$ is injective with finite dimensional cokernel, so by rescaling and inverse Fourier 
transform we deduce the injectivity, closed range and infinite dimensional cokernel of the normal operator 
$N(\Dirac): r^{\mu} H^1_e \to r^{\mu-1}L^2$. 

The parametrix construction in the edge calculus then yields an operator $G \in \Psi^{-1, \calE}_e(\calU \setminus \Sigma)$, 
where $\calE$ is an index family describing the asymptotic expansions of the Schwartz kernel of $G$ at each of the boundary 
faces of the edge double space associated to the blowup $\calU_{\Si}$. These expansions allow one to prove that this 
operator is bounded as a map $r^{\mu-1}L^2 \to r^\mu H^1_e$ and satisfies 
\[
G \circ \Dirac = I - R.
\]
This remainder term $R$ is a smoothing operator, hence maps $r^\mu H^1_e \to r^\mu H^n_e$ for all $n \geq 0$; 
more importantly, it has imagine in the space of polyhomogeneous distributions $\calA_{\phg}^{E} \subset r^\mu H^n_e$,
where $E$ is the index set for the expansion of $G$ along its `left' ($r \to 0$) face.  Thus if $\Dirac u = 0$, then applying
$G$ to both sides yields that $u = Ru$, from which we conclude that $u$ is polyhomogeneous.

By contrast, when $\mu \in (-1/2, 1/2)$, the Bessel operator is surjective but has nontrivial (finite dimensional) nullspace,
hence $N(\Dirac): r^{\mu}H^1_e \to r^{\mu-1}L^2$ is surjective but has infinite dimensional nullspace. We can then construct
a {\it right} parametrix $G'$ such that $\Dirac \circ G' = I - R'$, where $R'$ has range in polyhomogeneous sections, but
in this case the left parametrix $G$, $G \circ \Dirac = I - R$, leaves an error term which is smoothing but no longer maps
into polyhomogeneous sections. The weak asymptotic expansion which holds for solutions of $\Dirac u = 0$ in this case
is obtained in a somewhat indirect manner using successive meromorphic extensions of the Mellin transform, 
see \S 7 of \cite{Mazzeo91_EllTheoryOfDiffEdgeOp} for details. 
\end{proof}
One of the key conclusions of all of this is the existence of the Cauchy data map $\calB_k$.

\begin{definition}
There exists a bounded map
\[
\calB_0:  Z_0 \longrightarrow H^{-1/2}(\Si; \CC) \oplus H^{-1/2}(\Si; \CC)
\]
which sends any $L^2$ solution $u$ of $\Dirac u = 0$ to its pair of leading coefficients 
\[
\calB_0( u) = \begin{pmatrix} a_0(y) \\ b_0(y) \end{pmatrix}
\]
(which we typically write more simply as $\calB_0(u) = (a_0(y), b_0(y))$, where these leading coeficients 
are the ones appearing in the weak expansion of Proposition \ref{wexp}. Similarly, with $Z_1 = \mathrm{ker}\,(\Dirac|_{r L^2})$, there exists a bounded map of finite rank
\[
Z_1 \ni u \mapsto \calB_1(u) = \begin{pmatrix} a_1(y) \\ b_1(y) \end{pmatrix}  \in \calC^\infty(\Si; \CC) \oplus \calC^\infty(\Si; \CC),
\]
where now the leading coefficients are as in Proposition \ref{sexp}.
\end{definition} 


\subsection{The Calderon subspace}
We now consider the space of `Cauchy data' corresponding to solutions $u \in Z_0$, i.e., the image of the map $\calB_0$. 
This is called the Calderon subspace of $\Dirac$. 
The main result of this section is that this subspace $\calB_0(Z_0) \subset H^{-1/2} \oplus H^{-1/2}$ is the graph of
a certain bounded operator from $H^{-1/2}$ to $H^{-1/2}$; this operator is the analogue of the classical Dirichlet-to-Neumann 
operator in this edge setting (from a slightly different point of view, it is also the analogue of the scattering operator). 

\begin{proposition}
There exists an elliptic pseudodifferential operator $\calN \in \Psi^{0}(\Si, \CC)$ such that 
if $u \in Z_0$ and  $\calB(u) = (a_0, b_0)$, then 
\[
b_0 = \calN (a_0).
\]
This operator is self-adjoint on $L^2(\Si)$, and has principal symbol $\sigma_0(\calN)(y,\eta) = -i \heta = -i \eta/ |\eta|$;
this acts by Clifford multiplication.
\label{DtN}
\end{proposition}
\begin{remark}
Although $\calN$ is $\CC$-linear, it is necessary below to work in {\it real} Hilbert spaces, and thus 
regard $\calN$ as an $\RR$-linear operator only.
\end{remark}
\begin{corollary}
\label{cor:bpairing}
The range $\calB_0(Z_0) \subset H^{-1/2}(\Si; \CC) \oplus H^{-1/2}(\Si; \CC)$ is a closed subspace which we denote 
$\Lambda_{\Dirac}$.   Denote by the same symbol the intersection of this subspace with $L^2 \oplus L^2$; this is 
closed in $L^2 \oplus L^2$ since it is the graph of $\calN$ restricted to $L^2$.  Then $\Lambda_{\Dirac}$ 
is Lagrangian with respect to the natural $\RR$-linear pairing
\[
\omega\big( (a_0, b_0), (\tilde a_0, \tilde b_0)\big) =  \mathrm{Re}\, \int_{\Si}   \cl(\del_r) \begin{pmatrix} a_0 e^{-i\theta/2}\\ b_0 e^{i\theta/2} 
\end{pmatrix} \cdot  \overline{\begin{pmatrix} \tilde{a}_0 e^{-i\theta/2}\\ \tilde{b}_0 e^{i\theta/2} \end{pmatrix}} = 
- \mathrm{Im} \, \int_\Si  a_0(y) \overline{\tilde{b}_0(y)}  + b_0(y) \overline{\tilde{a}_0(y)}. 
\]
\end{corollary}
\begin{proof}
Fix $u, v \in Z_0$ and set $\calB_0(u) = (a_0, b_0)$, $\calB_0(v) = (\tilde{a}_0, \tilde{b}_0)$.  Suppose that both of these lie
in $L^2 \oplus L^2$.  By the self-adjointness of the Dirac operator on the {\it real} Hilbert space $L^2(M_\Si; \CC^2)$ 
with inner product
\[
\langle u, v \rangle = \mathrm{Re}\, \int u \cdot \overline{v}
\]
we have 
\[
0 = \langle \Dirac u, v \rangle - \langle u, \Dirac v \rangle = \lim_{\epsilon \to 0} \, \mathrm{Re}\, \int_{\{r=\epsilon\} }  
\cl(\del_r) u \cdot \overline{v} = \omega\big( (a_0, b_0),  (\tilde{a}_0, \tilde{b}_0)\big).
\]
This shows that $\Lambda_{\Dirac}$ is isotropic in $L^2 \oplus L^2$. 

Next, suppose that $u \in Z_0$ and $\calB_0(u) = (a_0, b_0) \in L^2 \oplus L^2$. Let $(\tilde{a}_0, \tilde{b}_0) \in L^2 \oplus L^2$ be some 
other pair such that $\omega\big( (a_0, b_0), (\tilde{a}_0, \tilde{b}_0) \big) = 0$.  By the continuity in $L^2$ of this bilinear form, it suffices to 
assume that $\tilde{a}_0$ and $\tilde{b}_0$ are both smooth. We claim that there exists an element $v \in Z_0$ such that $\calB_0(v )=
(\tilde{a}_0, \tilde{b}_0)$. Indeed, defining
\[
v_1 = \chi(r) \begin{pmatrix}  \tilde{a}_0 e^{-i\theta/2} \\ \tilde{b}_0 e^{i\theta/2} \end{pmatrix} r^{-1/2},
\]
where $\chi$ is a cutoff function which equals $1$ for $r$ sufficiently small and vanishes for $r \geq r_0$, then 
$v_1 \in L^2$ and by a short calculation, $\Dirac v_1 := f$ is bounded by $Cr^{-1/2}$, so $f \in L^2$ as well.
We still write $\calB_0(v_1 ) = (\tilde{a}_0, \tilde{b}_0)$. 

Now observe that $f \in Z_0^\perp$. Indeed, if $\phi \in Z_0$, then 
\[
\int f \phi = \int (\Dirac v_1) \phi = \int  (\Dirac v_1 ) \phi - v_1 (\Dirac \phi) = 
\omega\big( (a_0, b_0), (\tilde{a}_0, \tilde{b}_0)\big) = 0.
\]
The integration by parts is straightforward to justify.   We may now apply Theorem~\ref{mainmapD} to find an
element $w \in r H^1_e$ such that $\Dirac w = f$.  Finally, defining $v = \hat{v} - w$, we see that
$\Dirac v = 0$ and, since $|w| r^{1/2} \to 0$, we also have $\calB_0(v) = (\tilde{a}_0, \tilde{b}_0)$.  In other words,
the pair $(\tilde{a}_0, \tilde{b}_0) \in \Lambda_{\Dirac}$. 
\end{proof}

We now indicate some steps in the proof of Proposition \ref{DtN}.   As already explained, there
exists a generalized right inverse $G$ for the map $\Dirac: H^1_e \to r^{-1} L^2$. This operator satisfies
\[
\Dirac \circ G = I - \Pi_1, \qquad G\circ \Dirac = I - \Pi_2
\]
where $\Pi_1$ is the finite rank projector onto the cokernel $Z_1$ of $\Dirac$ in $r^{-1}L^2$ and $\Pi_2$ 
is the infinite rank projector onto the nullspace $Z_0$ of $\Dirac$ in $H^1_e$.   All of the operators $G, \Pi_1, \Pi_2$
are pseudodifferential edge operators, and as we have explained briefly above, they are constructed starting
from the corresponding operators for $B(\Dirac)$. 

In a similar vein, the boundary trace map $\calB_0$ for $\Dirac$ on $L^2$ is constructed using the corresponding 
boundary trace map for this model Bessel operator.  We have already described the leading behavior of $L^2$ 
solutions of $B(\Dirac)$, however: 
\begin{proposition}  If $\psi \in \calK(0) = \ker( B(\Dirac)|_{L^2})$, then there exists a complex number $\sigma$ such that
\[
\psi = \sigma \, \begin{pmatrix} \  e^{-i\theta/2}\\ - i \heta \, e^{i\theta/2}  \end{pmatrix}   t^{-\frac12}  e^{-t}.
\]
\label{Besselsoln}
\end{proposition} 

There is an analogous result for $N(\Dirac)$ which can be obtained from this result by rescaling and taking the inverse 
Fourier transform: 
\begin{proposition}
Suppose that $\psi \in L^2( \RR^3 \setminus \RR; \slS \otimes \calI)$ and $\Dirac \psi = 0$. The boundary
trace $\calB_0(\psi) = (a_0, b_0)$ is defined using Proposition \ref{wexp}. These leading coefficients lie in $H^{-1/2}(\RR)$ and satisfy
\[
b_0(y) = \calH ( a_0)(y).
\]
Here $\calH$ is the Hilbert transform, defined via its oscillatory integral representation
\[
\calH a_0 (y) = \int e^{i(y - \tilde{y})\eta} (-i \, \mathrm{sgn}\, \eta)\, a_0(\tilde{y}) \, d\tilde{y} d\xi.
\] 
This is a pseudodifferential operator of order $0$ with principal symbol $-i \heta = -i \eta/|\eta|$.  
\end{proposition}
\begin{proof}  One deduces immediately from the formula in Proposition~\ref{Besselsoln} that if $\psi \in L^2$ then 
\[
\psi ( s, \theta, w) = \int e^{i w \cdot \eta}  \sigma( \eta) \begin{pmatrix} e^{-i\theta/2} \\ i \heta e^{i\theta/2}\end{pmatrix}
s^{-1/2} e^{-s|\eta|}\, d\eta
\]
for some $\sigma(\eta)$. This shows that the Fourier transforms of $a_0$ and $b_0$ equal
$\sigma(\eta)$ and $-i\heta \sigma(\eta)$, respectively.  Furthermore, by the Plancherel theorem,
\[
\int |\psi|^2  s ds d\theta dw = 4\pi \int |\sigma(\eta)|^2  s^{-1} e^{-2s|\eta|} \, s ds d\eta = 
2\pi \int |\sigma(\eta)|^2 |\eta|^{-1} \, d\eta.
\]
This proves that $\sigma$ lies in the (homogeneous) Sobolev space $\dot{H}^{-1/2}(\RR)$ (for the 
ordinary Sobolev space, the Fourier multiplier would be smooth and bounded near $\eta = 0$). 

The characterization of $\calH$ is immediate from these formulas.
\end{proof}

The final step is to prove a similar theorem for $\Dirac$ itself. This requires a somewhat more elaborate
parametrix construction as carried out in \cite{MaVe}. 

\subsection{Boundary operators}
\label{Sect_BdryOp}
A general theory of elliptic boundary problems associated to edge operators such as $\Dirac$ is worked out
in detail in \cite{MaVe}, see also \cite{Usula}, \cite{Usula2}.  In the present setting, we consider problems 
of the form
\begin{equation}
\Dirac  u = f, \qquad \calT ( \calB_0(u)) = h,
\label{bvp}
\end{equation}
with $u$ and $f$ in $L^2$ and where $\calT$ is a differential or pseudodifferential operator acting on pairs $(a_0,b_0) = \calB_0(u)$.
For the present goals, we consider only a fairly special case, where $\calT$ imposes algebraic conditions on these pairs.

So far we only discussed the boundary trace $\calB_0(u)$ for solutuions of $\Dirac u = 0$. However, there is an extension of
Proposition \ref{wexp} for solutions of inhomogeneous equations $\Dirac u = f$.  
\begin{proposition}[{\cite[\S 7]{Mazzeo91_EllTheoryOfDiffEdgeOp}}]
If $u \in H^1_e$ and $f:= \Dirac u \in L^2$, then
\[
u = \begin{pmatrix} a_0(y) e^{-i\theta/2} \\ b_0(y) e^{i\theta/2} \end{pmatrix} r^{-1/2} + \tilde{u},
\]
where $a_0, b_0 \in H^{-1/2}(\Si)$, and in a weak sense $\tilde{u} = \mathcal O(r^{1/2})$. 
\end{proposition}
We refer to the cited paper for the precise meaning of the decay of the higher order term $\tilde{u}$.  This result is proved
using the same parametrix methods as described in Section 3.6. In any case, this provides a way to define $\calB_0(u) = (a_0, b_0)$. 
The $H^{-1/2}$ regularity of these leading coefficients is optimal; indeed, we already exhibited solutions $\Dirac v = 0$ with leading 
coefficients lying in $H^{-1/2}$ but no better space.    There is a slight strenghening which is proved in the same place. 
\begin{proposition}
Let $G$ be the generalized inverse for $\Dirac: L^2 \to r^{-1}L^2$, so $G \Dirac = I - \Pi_2$, where $\Pi_2$ is the orthogonal projection onto
the nullspace of $\Dirac$ in $L^2$. If $\Dirac u = f$ with $u \in H^1_e$ and $f \in L^2$, and if $\Pi_2 u = 0$, or even if we only know that
$\Pi_2 u$ is polyhomogeneous or has some fixed positive finite order of regularity at the boundary, then 
\[
\calB_0(u) = (a_0, b_0) \in H^{1/2}(\Si)^2. 
\]
\end{proposition}
This is a consequence of the mapping properties of $G$. Indeed, $u = Gf + \Pi_2 u$, and under the hypotheses, the second term on
the right has an expansion with more regular terms, so it suffices to analyze the expansion of $Gf$. This does have leading coefficients
in $H^{1/2}$. 

We now recall a generalization of the classical Calderon (Lopatinski-Schapiro) boundary condition which is necessary and sufficient 
for \eqref{bvp} to be well posed. From a functional analytic point of view, the requirement is that the restriction of $\calT$ 
to the Calderon subspace $\Lambda_{\Dirac}$ is Fredholm onto its image (where some suitable regularity is imposed on $h$).  
In this case, the boundary problem $(\Dirac, \calT)$ is called elliptic.   The beautiful fact is that this is equivalent to a much 
more easily verified condition involving the symbols of $\Dirac$ and $\calT$.   We shall state this more carefully momentarily 
in a special case. 

We next introduce the boundary conditions of interest.   Fix any pair of smooth functions $(c_1, d_1)$ along $\Si$, and assume
the nondegeneracy condition that $(c_1, d_1) \neq (0,0)$ everywhere along $\Si$.  To be concise with the notation, we write
$\psi = (c_1, d_1)$. This is an abuse of notation, but is a reminder that the motivating example is when $(c_1, d_1)$ is the pair of 
leading coefficients (of the $r^{1/2}$ term) of a $\Z_2$ harmonic spinor $\psi$. We say more about this just below.
In any case, define $\calT = \calT_{\psi}$ by 
\[
\calT(a_0, b_0) = \overline{d_1} a_0 - c_1 \overline{b_0}.
\]
This maps pairs $(a_0, b_0)$ to $\CC$-valued functions $h$, but the reader should be aware that this 
is only a {\it real} operator, and not $\CC$-linear!  

A fact which is useful in applications is that if $\calT(a_0, b_0) = 0$, i.e.,
\[
\overline{d_1} a_0 = c_1 \overline{b_0},
\]
then the function
\[
\zeta =  a_0/c_1 = \overline{ b_0 /d_1} 
\]
is well-defined on all of $\Si$, by virtue of the nondegeneracy hypothesis. 

Here is the motivation for this definition. Suppose that $\dim Z_1 = 1$ and let $\psi$ be a nonzero element in $Z_1$. 
We have noted that $\psi$ is polyhomogeneous, with expansion starting with $r^{1/2}$.  The coefficients of this first term
are thus well-defined, and we write 
\[
\calB_1(\psi) = (c_1, d_1) \in \calC^\infty(\Si)^2.
\]
Next, suppose that $\zeta$ is a section of the normal bundle $N\Si$; this is a trivial rank $2$ bundle, and is spanned 
in Fermi coordinates by coordinate vector fields $\del_{x_1}$ and $\del_{x_2}$. It is often helpful to write $z = x_1 + i x_2$
and think of $\zeta$ as being complex-valued. 

Any such section is an infinitesimal deformation of $\Si$. The deformation theory of $\Z_2$ harmonic spinors involves deforming
both $\Si$ and $\psi$.  A brief calculation shows that if $\Si$ is perturbed in the direction $\zeta$, then the corresponding infinitesimal 
deformation of the $\Z_2$-harmonic spinor $\psi$ has the form 
\[
\calL_\zeta \psi := 
\begin{pmatrix}  
	c_1 \zeta /\sqrt{z} \\ d_1 \overline{\zeta}/\sqrt{\overline{z}} 
\end{pmatrix} 
	+ \calO(r^{1/2}) = 
\begin{pmatrix}
	c_1\zeta e^{-i\theta/2} \\ d_1\bar\zeta e^{i\theta/2} 
\end{pmatrix} r^{-1/2} 
	+ \calO(r^{1/2}).
\]
We then seek some $\phi \in r H^1_e$ so that $\calL_\zeta \psi + \phi = u \in Z_0$. If this is possible, then at the infinitesimal level, one 
can compensate for the infinite dimensional cokernel $Z_0$ of $\Dirac: r H^1_e \to L^2$. Since this correction term $\phi$ does not 
affect the $r^{-1/2}$ coefficient, this is tantamount to asking whether a given pair $( c_1\zeta , d_1\bar\zeta)$ lies in the Calderon subspace $\Lambda_{\Dirac}$.   

Now note that the boundary operator $\calT_{\psi}$ satisfies 
\[
\calT_\psi ( c_1 \zeta, d_1 \overline{\zeta}) = \overline{d_1} \, ( c_1 \zeta) - c_1 \, \overline{ (d_1 \overline{\zeta})} = 0
\]
for every $\zeta$.  Conversely, the only pairs killed by $\calT_{\psi}$ necessarily have the form  $( c_1 \zeta, d_1 
\overline{\zeta})$ for some $\zeta$.   In particular, if $(c_1, d_1)$ are the leading coefficients $\calB_1(\psi)$ for $\psi \in Z_1$,
then a pair $(a_0, b_0) \in \mathrm{ker}\, (\calT_{\psi})$ must have the form $(c_1 \zeta, d_1 \overline{\zeta})$. Hence, by nondegeneracy, 
we can reconstruct the normal perturbation $\zeta$ of $\Si$.

We now return to the slightly more general setting where $\psi = (c_1, d_1)$ does not necessarily arise from a $\Z_2$ harmonic spinor. For simplicity, 
we abbreviate $\calT_{\psi}\circ \calB_0$ by simply writing it as $\calT_\psi$, i.e. we automatically assume that the boundary trace operator has been 
applied before $\calT_\psi$. 
\begin{proposition}
Assuming that the pair $(c_1, d_1)$ is nondegenerate, the boundary problem $(\Dirac, \calT_{\psi})$ is elliptic and self-adoint. 
\end{proposition}
\begin{proof}
We first explain what it means for a boundary problem to be elliptic. We attempt to solve, at least up to finite rank errors,
the problem $\Dirac u = f \in L^2$, $\calT_{\psi} u = h \in H^{-1/2}$. This is done in two steps. In the first, using the essential 
surjectivity of $\Dirac: H^1_e \to r^{-1}L^2$ as described in \autoref{Dmapprop}, we choose $v \in H^1_e$  such that $\Dirac v = f$ modulo $Z_1$.  In the second, this 
initial solution is modified by an element $w \in Z_0$ chosen so that $u = v-w$ solves the boundary condition. Obviously 
$\Dirac (v-w) = f - 0 = f$ modulo $Z_1$, and $\calB_0(u) = \calB_0(v) - \calB_0(w)$. Thus we must choose $w$ so that
\[
\calT_{\psi} w = \calT_{\psi} v - h. 
\]
If we can show that $\calT_{\psi}: \Lambda_{\Dirac} \to H^{-1/2}$ is Fredholm, then this is solvable modulo finite rank errors.

As noted earlier, Calderon proved that this condition is equivalent to a simpler finite dimensional statement involving the symbols of 
$\Dirac$ and $\calT_\psi$: the symbol of $\calT_\psi$ is bijective from the graph of the symbol of $\calN$. 

To state this finite dimensional criterion, recall that the symbol of $\calN$ is 
\[
\sigma_0(\calN)(y, \eta) = \begin{pmatrix} 1 \\ -i \heta \end{pmatrix}.
\]
On the other hand, consider $\calT_\psi$ as a real operator acting on the real and imaginary parts of $(a_0, b_0)$. As such, its symbol equals
\[
\sigma_0(\calT_\psi) =  \begin{pmatrix}  \overline{d_1} \otimes I  &   -c_1 \otimes {\bf c} \end{pmatrix},
\]
where $\bf c = \begin{pmatrix} 1 & 0 \\ 0 & -1 \end{pmatrix}$ corresponds to complex conjugation.   Thus to check ellipticity, 
we must analyze
\begin{equation}
\sigma(\calT_\psi) \circ \sigma(\calN) =  \begin{pmatrix}  \overline{d_1}  &   -c_1 {\bf c} \end{pmatrix} \,  \begin{pmatrix}  1 \\ -i \heta \end{pmatrix} = 
\overline{d_1} - c_1 {\bf c} (-i \heta).
\label{bsymb}
\end{equation}

\medskip

\noindent {\bf Claim:} This boundary operator $\calT_\psi$ is elliptic if and only if $|c_1|^2 + |d_1|^2$ is nowhere vanishing.

\medskip

Ellipticity is equivalent to showing that \eqref{bsymb} is invertible for every $y \in S^1$ and $\heta \in \{\pm 1\}$.   We prove this 
by showing that the composition of this endomorphism with the endomorphism $d_1 - i c_1 \heta {\bf c}$ is invertible. 

Recall first that we are regarding these objects as symbols, and as such they must obey the symbolic
version of the relationship
\[
\calH \circ {\bf c} = - {\bf c} \circ \calH
\]
acting on functions on $\RR$. This anticommutativity is straightforward to check from the basic definitions,
and uses the fact that if $\calF$ denotes Fourier transform, and $R$ the reflection map $Rf(y) = f(-y)$, then
$\calF \circ {\bf c} = R \circ {\bf c} \circ \calF$. The symbolic version of this statement is that
\[
{\bf c} ( \heta) = - \heta \, {\bf c}. 
\]

We can now compute
\[
\begin{aligned}
& \left(\overline{d_1} - c_1 {\bf c}( -i \heta) \right) \left( d_1 -i \heta c_1 {\bf c} \right) =  |d_1|^2 - i \heta \overline{d_1} c_1 {\bf c}
- c_1 {\bf c} ( -i \heta d_1)  + c_1 {\bf c}( -i \heta (-i c_1 {\bf c} )  \\
& = |d_1|^2 - i \heta \overline{d_1} c_1 {\bf c} + c_1 i \heta \overline{d_1} {\bf c}   + |c_1|^2  = |c_1|^2 + |d_1|^2.
\end{aligned}
\]
This proves the claim.

We now turn to proving the self-adjointness of $(\Dirac, \calT_\psi)$. Recall first the definition of the 
Hilbert space adjoint of the closed, unbounded operator $\Dirac$ on $L^2$ with domain
\[
\calD_{\psi} := \big \{ u \in L^2\mid  \Dirac u \in L^2,\ \mbox{and}\ \calT_{\psi} u = 0\big\}.
\]
By definition, the adjoint domain is defined by
\[
\calD_{\psi}^* = \Bigg\{v \in L^2\;\Big |\;  \mbox{for all $u \in \calD_{\psi}$, there exists some}\ C > 0\ \mbox{such that}\ 
\left| \int u \cdot \Dirac v \right| \leq C ||v||_{L^2} \Bigg\}.
\]
To understand this more concretely, observe that
\[
\int u \cdot \Dirac v = \int \Dirac u \cdot v - \omega\big ( \calB_0(u), \calB_0(v)\big ),
\]
so it follows that
\[
\calD_{\psi}^* = \big \{v \in L^2\mid \Dirac v \in L^2\ \mbox{and}\ \omega\big ( \calB_0(u), \calB_0(v)\big ) = 0\ \mbox{for all}\ u \in \calD_{\psi}\big \}.
\]

Consider, then, any pair $(a_0, b_0)$ (which we may as well take to be smooth) which satisfies
\[
\calT_{\psi} (a_0, b_0) = \overline{d_1} a_0 - c_1 \overline{b_0} = 0.
\]
Let $\calB_0(v) = (\tilde{a}_0, \tilde{b}_0)$.  We wish to show that if 
\[
\Im \int_\Si  a_0 \overline{\tilde{b}_0} + b_0 \overline{\tilde{a}_0} = 0
\]
for all such pairs $(a_0, b_0)$, then $\calT_{\psi} (\tilde{a}_0, \tilde{b}_0) = 0$, and conversely, if $\calT_{\psi}(\tilde{a}_0, \tilde{b}_0) = 0$, then this boundary 
pairing vanishes. Recall that if $\calT_\psi(a_0, b_0) = 0$, then there exists some complex-valued function $\zeta$ so that $a_0 = c_1 \zeta$, 
$b_0 = d_1 \overline{\zeta}$.  In terms of this,
\[
\omega\big( (a_0, b_0), (\tilde{a}_0, \tilde{b}_0)\big) = - \Im\, \int_\Si  c_1 \, \zeta \, \overline{\tilde{b}_0} + d_1 \, \overline{\zeta} \, \overline{\tilde{a}_0}.
\]
First suppose that $\zeta = f$ is real-valued.  Then
\[
\Im \, \int_\Si  f (c_1 \overline{\tilde{b}_0} + d_1  \overline{\tilde{a}_0}) = 0
\]
for all $f$, hence 
\begin{equation}
\Im \, (c_1 \overline{\tilde{b}_0} + d_1  \overline{\tilde{a}_0}) \equiv 0.
\label{Im1}
\end{equation}
If $\zeta = if$ is purely imaginary, then
\[
\Im \, \int_\Si if (c_1 \overline{\tilde{b}_0} - d_1 \overline{\tilde{a}_0}) = 0,
\]
again for all $f$. This implies that
\begin{equation}
\Re \, (c_1 \overline{\tilde{b}_0} - d_1 \overline{\tilde{a}_0}) = 0.
\label{Re1}
\end{equation}

We claim that \eqref{Im1} and \eqref{Re1} together imply that $\calT_{\psi}( \tilde{a}_0, \tilde{b}_0) = 0$. Indeed, write
$c_1 \overline{\tilde{b}_0} - \overline{d_1} \tilde{a}_0  = A$, so $\overline{d_1} \tilde{a}_0 = c_1 \overline{\tilde{b}_0} - A$. Inserting this
into these two equations gives
\[
\Im\, ( c_1 \overline{\tilde{b}_0}  + \overline{c_1} \tilde{b}_0 - A) = 0, \qquad \Re\,( c_1 \overline{\tilde{b}_0} - \overline{c_1} \tilde{b}_0 - A) = 0.
\]
However, if $w$ is any complex number, $\Im (w + \overline{w}) = 0$ and $\Re ( w - \overline{w}) = 0$, so these equations
reduce to $\Re\, A = \Im \, A = 0$, or finally $A = 0$.  
Thus $\calT_{\psi}(\tilde{a}_0, \tilde{b}_0) = 0$, and finally, $\calD_{\psi}^* = \calD_{\psi}$, as desired.
\end{proof}

The ellipticity of this pair $(\Dirac, \calT_{\psi})$ now implies, by the main theorem of \cite{MaVe}, that the boundary problem
\[
\Dirac: \calD_{\psi} \longrightarrow L^2
\]
is Fredholm. 

\section{Analysis of $\Dirac$ when $\Si$ is an admissible graph}
We now turn to the case where $\Si = \calP \cup \calE$ is an admissible graph. Certain parts of the previous theory are local, and hence they
carry over to the interiors of each edge. We shall bring these in as needed.  One of these we mention right away, although it will
not be needed until later. This is the fact that there is still a boundary trace map 
\[
\calB_0: L^2(M; \slS \otimes \calI) \longrightarrow \bigsqcup_{e \in \calE} H^{-1/2}_{\mathrm{loc}}( \mathrm{int}\, e; \CC)^2.
\]

The main new step is to analyze $\Dirac$ near each vertex $q \in \calP$. This proceeds by a study of the
model problem, i.e., the Dirac operator on the exact metric cone $C(S_q)$. Here $S_q$ is the boundary face of $M_\Si$ 
created by blowing up $q$; recalling that the edges of $\Si$ are also blown up in $M_\Si$, $S_q$ is a surface with
boundary obtained by blowing up the $2$-sphere around $2k$ punctures, where $2k = 2k_q$ is the valence of that vertex. 
For simplicity we drop the subscript $q$ and restrict to a neighborhood $\calU$ of $S$ in $M_\Si$.

There exists a spherical coordinate system $(\rho, \omega)$ in $\calU$ where each edge $e_j$ which abuts $q$ corresponds
to a ray $\{\omega = p_j \in S^2,\ \rho \geq 0\}$.  This fixed $2k$ points $\{p_1, \ldots, p_{2k}\}$ on the normal sphere at $q$
which we call the punctures, and hence $2k$ different boundary circles of $S$.  If the background metric is flat in $\calU$, then 
$\Dirac$ takes the particularly simple form
\begin{equation}
\Dirac  = \mathtt{cl}(\del_\rho) \left( (\del_\rho + \frac1{\rho}) - \frac{1}{\rho}\widehat{\Dirac}\right),
\label{Diracsphcoor}
\end{equation}
where $\widehat{\Dirac}$ is the induced Dirac operator on the (punctured) sphere.  For more general metrics, this local
coordinate expression for $\Dirac$ is the sum of this model operator in \eqref{Diracsphcoor} and an error term $E$ which 
is lower order as $\rho \to 0$. Since our eventual goal is to compute an index, it is sufficient to study the exact model 
conic operator, which avoids some minor technical issues. 

\subsection{The twisted Dirac operator $\widehat{\Dirac}$ on a punctured sphere} 
The expression in \eqref{Diracsphcoor} is that of an elliptic conic operator, and as for any such operator, its analysis 
draws on the analytic properties of the `tangential', or cross-sectional, operator $\widehat{\Dirac}$, which acts on
the sections of the restriction of the twisted spin bundle to the punctured $2$-sphere. 
A key goal is to show that this operator, with boundary conditions at the boundary 
circles induced from the boundary operator $\calT_\psi$ along the incoming edges, is self-adjoint with discrete spectrum.

We begin by noting that $\widehat{\Dirac}$ is a familiar object. Clifford multiplication by $\mathtt{cl}(\del_\rho)$ 
splits $\slS$ as $\slS^+ \oplus \slS^-$.  In this low dimension, these are identified with $K^{1/2}$ and ${\overline K}^{1/2}$, 
respectively, where $K$ is the canonical bundle of the $2$-sphere.  Using these identifications, 
\[
\widehat{\Dirac} = \begin{pmatrix} 0 &  - \del \\ \overline{\del} & 0 \end{pmatrix}
\]
acting on sections of $(K^{1/2} \otimes \calI) \oplus (\overline{K}^{1/2} \otimes \calI)$.  One way to view the twist by $\calI$ is that we are  
letting $\widehat{\Dirac}$ act on untwisted spinors on the double cover of $S^2$ ramified at the $q_j$, with the lifted metric; in other words, 
we can also regard $\widehat{\Dirac}$ as the untwisted Dirac operator on a space with conic singularities with cone angle $4\pi$ and with
an involution $\tau$, acting on ordinary spinors which are odd with respect to this involution. 
Denote this double-cover of $S^2$, ramified at the $q_j$, by $Y$. It is a smooth surface of genus $k-1$. The bundle $K^{1/2}_S\otimes \calI$ 
lifts simply to $K^{1/2}_Y$.   This description is useful in understanding the nullspace of $\widehat{\Dirac}$. 

Near each $q_j$, $\widehat{\Dirac}$ itself is conic (this is why we call $\Dirac$ on $M\setminus \Si$ an `iterated conic operator').  
The twisting bundle $\calI$ restricts to a flat $\RR$-bundle on $S^2 \setminus \{q_1, \ldots, q_{2k}\}$ with monodromy $-1$ around 
each $q_j$, which is possible by the admissibility hypothesis. 

In polar coordinates around $q_j$, 
\[
\widehat{\Dirac} = \mathtt{cl}(\del_r) + \frac{1}{\sin r} \mathtt{cl}(\del_\theta) + \mbox{higher order terms}.
\]
In the following, we first describe the local asymptotic behavior near $q_j$ of solutions to $\widehat{\Dirac} v = 0$, and
then use this to describe the minimal and maximal domains of $\widehat{\Dirac}$, as well as two important domains
on which this operator is self-adjoint. 

\medskip

\subsection{Indicial roots and regularity}
Fix a singular point $q_j$, and the corresponding boundary circle $C_j \subset \del S$. We define the indicial roots of $\widehat{\Dirac}$ at $C_j$ 
as before: $\lambda$ is an indicial root if there exists a section $w(\theta)$ such that $\widehat{\Dirac}( r^\lambda w(\theta)) = \calO(r^\lambda)$
(rather than the expected rate $\calO(r^{\lambda-1})$), i.e., there is a cancellation of leading order terms. The calculations here are 
{\it precisely} the same as in the computation of indicial roots of $\Dirac$ at the edges.  The result is that the set of indicial roots of
$\widehat{\Dirac}$ at each $q_j$ equals $\ZZ + \frac12$, and the corresponding (approximate) solutions 
are $\begin{pmatrix} e^{i(k-\frac12)\theta} \\ e^{i(k+\frac12)\theta} \end{pmatrix} r^{(k-1/2)}$.  In this two-dimensional setting, there is
no analogue of the normal operator.  We index these indicial roots by 
\[
\lambda_k = k - \frac12.
\]

Regularity for solutions of elliptic conic operators is much simpler than for edge operators; in particular, nullspace elements
are always polyhomogeneous.
\begin{proposition}
Any solution of $\widehat{\Dirac} v = 0$, or more generally of $\widehat{\Dirac} v = \nu v$, for any $\nu \in \RR$, with
$|v| \leq C r^{-N}$ for some $N$ at each $q_j$, is necessarily polyhomogeneous, with 
\[
v \sim  \sum_{k \geq k_0} \sum_{\ell = 0}^{k-k_0} r^{k-1/2}(\log r)^\ell \begin{pmatrix} e^{(k-1/2)i\theta} \hat{a}_{k,\ell}^{j} \\ e^{(k+1/2)i\theta} \hat{b}_{k,\ell}^{j} \end{pmatrix}
\]
for the integer $k_0$ such that $-N \leq k_0 < -N+1$.  If $v \in L^2$ then $k_0 \geq 0$.   
\label{conicregularity}
\end{proposition}
For simplicity, we denote the leading coefficient pair of a $\Z_2-$ harmonic spinor on $S$ by $(\hat{a}^{j}_{k_0},\hat{b}^{j}_{k_0})$. (A 
short calculation shows that since this expansion solves the equation formally, the leading term cannot have a logarithmic factor.)  
The main cases below and when $k_0 = 0$ or $1$. 

This pair of leading asymptotic coefficients is the boundary trace of $v$, and we write
\[
\widehat{ \calB}_{0}: v \longmapsto \big((\hat{a}^j_{0}, \hat{b}^j_{0})\big)_{j=1}^{2k},
\]
and $\widehat{ \calB}_{0}^j$ for the pair $(\hat{a}^j_0,  \hat{b}^j_0)$. 

\subsection{Mapping properties}  
We first consider the map 
\begin{equation}
\widehat{\Dirac}:  r^{\mu} H^1_b(S, \slS \otimes \calI; r drd\theta) \longrightarrow r^{\mu-1}L^2(S, \slS \otimes \calI; r dr d\theta)
\label{DSwss}
\end{equation}
between weighted $b$ Sobolev spaces.  (We recall the definition of $r^\mu H^1_b$ which appeared in Remark \ref{defH1b}.) 

As in \eqref{indrindw}, $r^\lambda w(\theta) \in r^\mu L^2$ (for $w \in \calC^\infty$)  in $r \leq 1$ if and only if $\mu < \lambda + 1$.
Thus for each indicial root $\lambda_k = k - \frac12$ we associate the indicial weight $\mu_k = k + \frac12$. 

\begin{proposition}
The map \eqref{DSwss} is Fredholm if and only if $\mu \neq \mu_k$ for any $k \in \ZZ$.   Assuming that $\mu$ is nonindicial,
the nullspace and cokernel of \eqref{DSwss} are naturally identified with the cokernel and nullspace, respectively, of 
\[
\widehat{\Dirac}: r^{1-\mu}H^1_b(S, \slS \otimes \calI; r dr d\theta) \longrightarrow r^{-\mu} L^2(S, \slS \otimes \calI; r dr d\theta).
\]
\label{dualityprop}
\end{proposition}
\begin{proof}
The first assertion is a standard result about elliptic $b$ operators. The remaining part is proved using the pairing
\[
r^\mu H^1_b \times r^{1-\mu} H^1_b \ni  (v, w) \longmapsto \langle \widehat{\Dirac} v, w \rangle = \langle v, \widehat{\Dirac} w\rangle,
\]
where $\langle \cdot, \cdot \rangle$ is the usual $L^2$ pairing extended to $r^{\mu-1}L^2 \times r^{1-\mu}L^2$.
\end{proof}
For any nonindicial weight $\mu$, write $K_{\mu}$ and $C_\mu$ for the dimensions of the kernel and cokernel of \eqref{DSwss}, 
and $\Ind(\mu)$ for its index. Using this duality, we have 
\begin{corollary} 
The index satisfies $\Ind(1-\mu) = - \Ind(\mu)$. 
\end{corollary}

The map \eqref{DSwss} is not Fredholm when $\mu = 1/2$.  However, there is a relative index theorem which calculates 
the change of the index across this (or any other) indicial weight. 
\begin{proposition}
For any $0 < \epsilon < 1$, 
\[
\Ind\Big(k + \frac12 + \epsilon\Big) - \Ind\Big(k + \frac12 - \epsilon\Big) = -4k.
\]
\label{relindex}
\end{proposition}
\begin{proof}
This is \cite[Theorem 6.5]{Melrose93_AtiyahPatodiSingerBook}.  The complete statement is that the difference in indices equals the negative of the sum 
over all $q_j$ of the multiplicity of the indicial root at that puncture. In this setting there are $2k$ points $q_j$, and each 
indicial root has multiplicity $2$, so the relative index is $-4k$.
\end{proof}

Now recalling the duality above and specializing to $\ell = 0$, we compute that
\begin{align*}
-4k = \Ind\Big( \frac12 + \epsilon\Big) - \Ind\Big(\frac12 - \epsilon\Big) & = (\dim K_{\frac12 + \epsilon} - \dim C_{\frac12 + \epsilon}) - (\dim K_{\frac12-\epsilon} - \dim C_{\frac12 -\epsilon}) \\
& = 2 (\dim K_{\frac12 + \epsilon} - \dim C_{\frac12 + \epsilon}),
\end{align*}
hence
\begin{corollary}
\[
\Ind\Big(\frac12 + \epsilon\Big) = -2k, \quad \Ind\Big(\frac12 - \epsilon\Big) = 2k.
\]
This implies that the space of solutions asymptotic to $r^{-1/2}$ (which correspond to the indicial weight $1/2$),
but which do not decay faster, has dimension equal to $2k$.
\label{jumpind}
\end{corollary}
There is no value $\mu$ such that $\Ind(\mu) = 0$. In other words, \eqref{DSwss} is never invertible. As $\mu$ increases, the kernel 
decreases in dimension and the cokernel becomes larger, with jumps occuring precisely at the indicial weights.  The dimension
count for this nullspace is obtained in a different way in Lemma \ref{dim_NullKr}.

\medskip

We next discuss the various choices of Hilbert space domain for $\widehat{\Dirac}$.  The most basic choices are the
minimal and maximal domains: 
\begin{definition} 
\[
\calD_{\min}(\widehat{\Dirac}) := \Big\{ v \in L^2\Big|\ \exists\, v_j \in \calC^\infty_0: \ v_j \overset{L^2}\to v,\ \widehat{\Dirac} v_j \overset{L^2}\to \widehat{\Dirac} v\Big\}
\]
and 
\[
\calD_{\max}(\widehat{\Dirac}) := \Big\{v \in L^2\Big|\ \widehat{\Dirac} v \in L^2\Big\}.
\]
\end{definition}
The characterization of these relies on a regularity result from \cite{Mazzeo91_EllTheoryOfDiffEdgeOp}.
\begin{proposition}
If $v \in \calD_{\max}$, then near each $p_j$, 
\[
v = \begin{pmatrix} \hat{a}_0^{j} e^{-i\theta/2} \\ \hat{b}_0^{j}\,  e^{i\theta/2} \end{pmatrix} r^{-1/2} + \tilde{v}, \qquad \tilde{v} \in r H^1_b, 
\]
and conversely, any function of this form lies in $\calD_{\max}$.  On the other hand, $\calD_{\min} = r H^1_b$ consists of 
those terms in $\calD_{\max}$ where this $r^{-1/2}$ coefficient vanishes.  

In particular, the boundary trace map
\[
\widehat{\calB}_0: \calD_{\max}(\widehat{\Dirac}) \rightarrow \CC^{4k}, \ \  v \mapsto \big((\hat{a}^j_0,\hat{b}^j_0)\big)_{j=1}^{2k}
\]
is well-defined and surjective.
\label{maxdomdecomp} 
\end{proposition}
The statement about the minimal domain is proved by showing that the term involving $r^{-1/2}$ is never approximable in the 
graph norm by smooth compactly supported sections. 

As a consequence of this Proposition, the quotient space $\calD_{\max}(\widehat{\Dirac})/\calD_{\min}(\widehat{\Dirac})$ is isomorphic
to $\CC^{4k}$.  A classical theorem of von Neumann states that the domains $\calD$ which contain the core domain $\calC^\infty_0$ 
and such that $\widehat{\Dirac} : \calD \to L^2$ is closed are in one-to-one correspondence with subspaces of this quotient.
This result is particularly elementary since the quotient is finite dimensional in this conic setting.

\subsection{Boundary pairing and self-adjoint extensions} 
The quotient $\calD_{\max}/\calD_{\min}$ has another important feature, namely a symplectic form
\[
\omega:  \calD_{\max}/\calD_{\min} \times \calD_{\max}/\calD_{\min} \longrightarrow \RR.
\]
This is defined initially via 
\[
\omega([u],[v])=\mathrm{Im}\sum_{j=1}^N\Bigg(\int_{S_j}\slD u\cdot v- u\cdot \slD v\Bigg),
\]
but upon using the decompositions of $u$ and $v$ from \eqref{maxdomdecomp} and using the same calculations as in
Corollary \ref{cor:bpairing}, it can be rewritten as
\begin{equation}
\omega \Big( \big(( \hat{a}_0^{j},\hat{b}_0^{j})\big)_{j=1}^{2k}, \big(( \tilde{a}^{j}_0,\tilde{b}^{j}_0)\big)_{j=1}^{2k} \Big) =  - \mathrm{Im}\, \sum_{j=1}^{2k}  \left(\hat{a}_0^{j} \overline{ \tilde{b}^{j}_0 } + \hat{b}_0^{j} \overline{ \tilde{a}^{j}_0} \right).
\label{sform}
\end{equation}

If $\calD_{\min} \subset \calD \subset \calD_{\max}$ is any one of these closed domains, there is a Hilbert space adjoint domain
$\calD^*$ which consists of all $v \in L^2$ such that the map $\calD \ni u \mapsto \langle \widehat{\Dirac} u, v \rangle$ 
extends to a bounded linear map defined for all $u \in L^2$.  We thus define, for any such $v \in \calD^*$, the 
adjoint operator $\widehat{\Dirac}^* v$ to be the unique element in $L^2$ such that this bounded linear functional equals
$u \mapsto \langle u, \widehat{\Dirac}^* v \rangle$.  It is easy to see that $\calD_{\min} \subset \calD^*$ and
$\widehat{\Dirac}^* = \widehat{\Dirac}$ on $\calD_{\min}$.   In fact, the minimal and maximal domains are dual to one another,
\[
(\widehat{\Dirac}, \calD_{\min})^* = (\widehat{\Dirac}, \calD_{\max}),\ \ (\widehat{\Dirac}, \calD_{\max})^* = (\widehat{\Dirac}, \calD_{\min}),
\]
and more generally still, using \eqref{sform}, 
\[
(\widehat{\Dirac}, \calD)^* = (\widehat{\Dirac}, \calD^\perp),
\]
where $\calD^\perp$ denotes the $\omega$-orthogonal complement.  (Strictly speaking, we take the preimage in $\calD_{\max}$
of the $\omega$-orthogonal complement of $\calD/\calD_{\max}$.) 

A particular closed extension $(\widehat{\Dirac}, \calD)$ is self-adjoint if it is equal to its own dual. By the remarks above,
this occurs precisely when $\calD^\perp = \calD$, or in other words, when $\calD$ descendes to a Lagrangian 
subspaces of $(\calD_{\max}/\calD_{\min}, \omega)$.   There are two important types of distinguished Lagrangian
subspace of this quotient, which correspond to two natural self-adjoint realizations of $\widehat{\Dirac}$.

The {\bf Krein extension} $\calD_{\Kr}$ is the extension of $\widehat{\Dirac}$ from $\calD_{\min}$ to $\calD_{\min} + 
\ker (\widehat{\Dirac}|_{\calD_{\max}})$.  This is actually a very general construction which yields a self-adjoint extension 
for any symmetric operator. It is a short exercise to check that it is its own dual.  The nullspace of $(\widehat{\Dirac}, \calD_{\Kr})$ 
consists of exactly those $v \in \calD_{\max}$ such that $\widehat{\Dirac} v = 0$, and the associated Lagrangian equals
the range of $\calB_0$ restricted to the nullspace of $\widehat{\Dirac}$ on $\calD_{\max}$.  In particular, the collection
of leading coefficients $\big(\hat{a}_0^{j},\hat{b}_0^{j})\big)_{j=1}^{2k}$ corresponding to nullspace elements $v \in \calD_{\max}$
is a complex $2k$-dimensional subspace of $\CC^{4k}$. 

The other interesting self-adjoint extensions of $\widehat{\Dirac}$ are induced by the boundary operators $\calT_{\psi}$ on
the edges of $\Si$.  Here $\psi := \{(c_1^j, d_1^j)\}_{j=1}^{2k}$ is a collection of pairs of functions along the edges of $\Si$,
and the boundary operator is defined using these exactly as in the case without vertices.
\begin{definition}\label{nondeg}
We say that this collection of functions, or the corresponding boundary operator $\calT_\psi$, is {\bf nondegenerate} if both of the 
following conditions hold:
\begin{itemize}
\item For each $j$, $(c_1^j, d_1^j) \neq (0,0)$ at any interior point of the edge $e_j$.
\item For every $j$, the limit at each end of the edge $e_j$
\[
\lim \dfrac{(c_1^j, d_1^j)}{| (c_1^j, d_1^j) |}
\]
exists.
\end{itemize}
\end{definition}
Denote by $(\hat{c}_1^j, \hat{d}_1^j)$ the limit of $\dfrac{(c_1^j, d_1^j)}{| (c_1^j, d_1^j) |}$ at $p$, $j=1,2,...,2k$. In terms of these limits, we can define the induced boundary operator $\widehat{\calT}_{\psi}$ at the face $S_p$ over the vertex $p$ by
\[
\widehat{\calT}_{\psi}\big(( \hat{a}_0^{j},\hat{b}_0^{j})\big)_{j=1}^{2k}\big):=\big(\overline{\hat{d}^{\, j}_1}\hat{a}^{\, j}_0-\hat{c}^{\, j}_1\overline{\hat{b}^{\, j}_0}\big)_{j=1}^{2k}
\]
for any $( \hat{a}_0^{j},\hat{b}_0^{j})\big)_{j=1}^{2k} \in \CC^{4k}$. This defines an $\R$-linear map $\widehat{\calT}_{\psi}:\CC^{2k}\times \overline{\CC^{2k}}\rightarrow \CC^{2k}$
which is block diagonal on $\big((\hat{a}^j_1,\hat{b}^j_1) \big)_{j = 1}^{2k}$. Define $\widehat{\calT}_{\psi}\circ \widehat{\calB}_0$ (which as before we 
abbreviate and refer to simply as $\widehat{\calT}_{\psi}$).  The same calculation as we gave at the end of Section 3 leads to the 
\begin{proposition}
The operator $\widehat{\Dirac}$ acting on $\widehat{\calD}_\psi = 
\{u \in \calD_{\max}|\  \widehat{ \calT}_{\psi} u = 0\}$ is self-adjoint. 
\end{proposition}

The operator $\widehat{\Dirac}$ on either $\calD_{\Kr}$ or any one of these other domains $\widehat{\calD}_\psi$ is
self-adjoint, and has compact resolvents. Thus these operators all have discrete spectra.  The Krein spectrum plays 
no role later. The eigenvalues of $(\widehat{\Dirac}, \widehat{\calD}_\psi)$ do enter our considerations below, and
we enumerate these as $\{\lambda_\ell \}$. 
\begin{proposition}\label{symm_of_spec}
The spectrum of $(\widehat{\Dirac}, \widehat{\calD}_\psi)$ is symmetric about $0$, and the dimension of the nullspace of this operator
is an even number $2h$ for some $h \in \NN$.
\end{proposition}
This follows immedately from the observation that the action of $\mathtt{cl}(\del_\rho)$ on $\slS^+ \oplus \slS^-$ on $S^2$ 
anticommutes with $\widehat{\Dirac}$, i.e., $\mathtt{cl}(\del_\rho) \widehat{\Dirac} + \widehat{\Dirac} \mathtt{cl}(\del_\rho) = 0$. 
This shows that if $\widehat{\Dirac} v = \lambda v$, then $\widehat{\Dirac} (\mathtt{cl}(\del_\rho) v) = - \lambda (\mathtt{cl}(\del_\rho) v)$. 
Furthermore, $\mathtt{cl}(\del_\rho)$ is a complex structure on the nullspace, so this nullspace has even dimension.

\subsection{Nullspaces}   We now describe the nullspaces of these self-adjoint operators.

\begin{lemma}\label{dim_NullKr}
The dimension of the nullspace of $(\widehat{\Dirac}, \calD_{\Kr})$ equals $2k$.
\end{lemma}
\begin{proof}
It suffices to check that the restriction of $\widehat{\calB}_0$ to this nullspace is injective.  If $v \in \calD_{\max}$ 
and $\widehat{\calB}_0(v) = 0$, then $v \in r H^1_b$.  This justifies the integration by parts
\[
||\widehat{\Dirac} v||^2 = \langle \widehat{\Dirac}\ ^2 v, v \rangle = ||\nabla v||^2 + \frac14 \langle R v, v\rangle;
\]
the second equality is the standard Weitzenb\"ock formula for $\widehat{\Dirac}$ on $S^2$. The scalar curvature $R$ of $S^2$ is $2$, 
so $||\widehat{\Dirac} v||^2 \geq \frac12 ||v||^2$, and hence if $\widehat{\Dirac} v = 0$, we must have $v \equiv 0$. 
Thus, as remarked earlier in fact, the Krein nullspace is isomorphic to the image of $\widehat{\calB}_0$, which has 
complex dimension $2k$ since it is Lagrangian.
\end{proof}

For the other extensions associated to boundary operator $\widehat{ \calT}_{\psi}$, the nullspace is a {\it real} subspace of $\CC^{4k}$,
but its dimension may vary over a range. 

\begin{proposition}\label{iso_of_null}
Suppose that $(\hat{c}_1^{\, j}, \hat{d}_1^{\, j}) \neq (0,0)$ for any $j$. Then the nullspace of $(\widehat{\Dirac}, \widehat{\calD}_\psi)$ 
has real dimension $2h$ for some $h\leq k$. Furthermore, $h=0$ for a dense open set of coefficients $\{(\hat{c}^j_1, \hat{d}^j_1)\}_{j=1}^{2k}
\in \CC^{4k}$. 
\end{proposition}
\begin{proof}
First observe that 
\[
\mathcal{D}_{\Kr}= \mathcal{D}^+\oplus\mathcal{D}^-, \qquad 
\mathcal{D}^{\pm}:=\{v\in\mathcal{D}_{\Kr}:\ v(x)\in \slS^{\pm} \mbox{ for all }x\}.
\]
By Proposition \ref{symm_of_spec} and Lemma \ref{dim_NullKr}, $\ker(\widehat{D}|_{\mathcal{D}^+})\cong \ker(\widehat{D}|_{\mathcal{D}^-})$,
and each has complex dimension $k$. Therefore, 
\begin{align}
\widehat{\calT}_\psi: \ker(\widehat{D}|_{\cal{D}^+})\oplus\ker(\widehat{D}|_{\cal{D}^-})\rightarrow \mathbb{C}^{2k}
\end{align}
is a $\R$-linear map $\RR^{4k} \longrightarrow \RR^{4k}$. From the definition of $\widehat{\calT}_\psi$, we see that if
$(v^+,v^-)\in \ker(\widehat{\calT}_\psi)$ then $(iv^+,iv^-)\notin \ker(\widehat{\calT}_\psi)$.   Thus 
any real linearly independent set $\{(v^+_j,v^-_j)\}_{j=1}^m\subset \ker(\widehat{\calT}_\psi)$ corresponds to a real linearly 
independent set $\{(v^+_j,v^-_j)\}_{j=1}^m\cup \{(iv^+_j,iv^-_j)\}_{j=1}^m$. 
Thus 
\[
\dim_{\R} \ker(\widehat{\calT}_\psi) \leq \frac12 \dim_{\R} \left(\ker(\widehat{D}|_{\mathcal{D}^+})\oplus\ker(\widehat{D}|_{\mathcal{D}^-}) \right) = 
\frac12 (2k + 2k) = 2k.
\]

If all of the $\hat{c}_1^{\, j}$ and $\hat{d}_1^{\, j}$ are nonvanishing, which is the generic case, then we may as well assume
that every $\hat{c}_1^{\, j}=1$. By Lemma \ref{dim_NullKr}, the nullspace of $(\widehat{D},\widehat{\calD}_{\psi})$ is 
the intersection of two real $2k$-dimensional spaces
\[
\Big\{(\overline{\hat{d}^{\, j}_1}\hat{a}^{\, j}_0)_{j=1}^{2k}\Big|(\hat{a}^{\, j}_0)\in\widehat{\calB}_0(\calD^+)\Big\}
\ \ \mbox{ and } \ \ 
\Big\{(\overline{\hat{b}^{\, j}_0})_{j=1}^{2k}\Big|(\hat{b}^{\, j}_0)\in\widehat{\calB}_0(\calD^-)\Big\}
\]
in $\CC^{2k}\cong\RR^{4k}$. This is trivial for generic $\hat{d}_1^{\, j}$.
\end{proof}

\begin{remark} 
It is also possible to express the elements in the kernel of $(\widehat{D}, \calD_{Kr})$ in holomorphic terms. 

To do so, first note that the spin bundle on $S^2$ is identified with $K^{1/2} \oplus K^{-1/2}$, where $K$ is the canonical bundle 
of the sphere. In terms of this, 
\[
\widehat{\Dirac} = \begin{pmatrix}  0 & \del \\ \overline{\del} & 0 \end{pmatrix}.
\]
Thus $\del (f dz^{-1/2}) = f_z dz^{1/2}$, while for the second row, we further identify $K^{1/2}$ with $\overline{K}^{-1/2}$ 
and $\overline{K}^{1/2}$ with $K^{-1/2}$, so that 
\[
\overline{\del} (f dz^{1/2}) = \overline{\del} (f d\bar{z}^{-1/2}) = f_{\bar{z}} d\bar{z}^{1/2} = f_{\bar{z}} dz^{-1/2}. 
\]
We conclude from this that a solution of $\widehat{\Dirac} v = 0$ is a pair
\[
v = \begin{pmatrix}  f_1 dz^{1/2} \\ f_2 d\bar{z}^{1/2} \end{pmatrix}\ \ \mbox{where}\ \ \ \del_{\bar{z}} f_1 = 0,\ \del_z f_2 = 0.
\]
The kernel of the boundary operator consists of all such pairs $(f_1, f_2)$ as above which are $\calO(r^{-1/2})$ at each $p_j$.  

First recall a more standard computation.  Consider the space of meromorphic abelian differentials on $S^2 = \mathbb P^1$ 
with at most simple poles at the points $p_1, \ldots, p_{2k}$.   In an affine holomorphic coordinate $z$, and
positioning the points $p_j$ away from infinity, any such differential $\omega = h dz$ where $h(z)$ is meromorphic
with simple poles at each of the $p_j$. To ensure that $\omega$ is regular at infinity, $h$ must also decay at least like $|z|^{-2}$ at infinity. 
These conditions imply that $h(z) = P(z)/ (z-p_1) \ldots (z-p_{2k})$, where $P(z)$ is a polynomial of order at most $2k-2$.
This proves that there is a (complex) $(2k-1)$-dimensional space of these solutions.  Similarly, the space of antiholomorphic forms
$k(z) d\bar{z}$ also has complex dimension $2k-1$. 

Next, a meromorphic differential $\omega$ with simple poles is the square of a section $f(z) \sqrt{dz}$ of $K^{1/2} \otimes \calI$ 
if and only if $P(z)$ is the square of another polynomial $Q(z)$, which thus has degree less than or equal to $k-1$: 
\[
f(z) = \frac{ Q(z)}{ \sqrt{(z-p_1) \ldots (z-p_{2k})} },\ \ \mathrm{deg}\, Q \leq k-1,
\]
with a similar computation for the antiholomorphic part. Thus $\dim_{\C} \ker(\widehat{D}|_{\calD^\pm }) = k$. 
\end{remark}

\subsection{The Dirac operator on a cone over a punctured sphere}
Equipped with this information about the `cross-sectional operator' $\widehat{\Dirac}$, we now return to the study of $\Dirac$ 
near a vertex $q$. In this subsection we assume for simplicity that the metric is Euclidean and that the graph $\Sigma$ is a union 
of rays emanating from the origin, so $\Dirac$ is the `standard' constant coefficient Dirac operator on $\RR^3$ (twisted by 
the $\RR$-bundle $\calI$): 
\begin{equation}
\Dirac = \mathtt{cl}(\del_\rho) \left( \left(\frac{\del\,}{\del \rho} + \frac{1}{\rho}\right) - \frac{1}{\rho}\widehat{\Dirac}\right).
\label{Dircps}
\end{equation}
This is again a conic operator, but one where the cross-section is no longer a closed manifold, but instead a manifold with boundary,
or equivalently, the punctured sphere.  The rays emanating from the origin corresponding to each puncture in $S^2$
are edges, as studied in Section 3. 

In this iterated edge setting, there is no reasonable characterization of the maximal domain of $\Dirac$. More specifically,
$\calD_{\max}(\Dirac)$ is defined as above, but the typical elements in it may have very complicated regularity near the vertices.
For example, there is no reason that these elements must have partial expansions, as we saw for isolated conic singularities.
We work instead with an {\it intermediate} domain which uses what we have learned about the induced Dirac operator
on the cross-sections. 

By Fubini's theorem, any $u \in L^2 (C(S))$ is a vector-valued function $u(\rho, \cdot) \in L^2(S; d\omega)$ with 
$L^2$ dependence on $\rho$:  
\[
L^2( \{(\rho, \omega) \in C(S)\}; \rho^2 d\rho d\omega) =  L^2( \RR^+, \rho^2 d\rho;  L^2(S; d\omega)).
\]
Restricting the cross-sectional space $L^2(S)$ and allowing regularity in the $\rho$ direction, as well as a weight function, we 
obtain the family of spaces
\[
\calD_{\psi, \mu} = \{ u \in \calD_{\max}(\Dirac)|\  u \in \rho^\mu H^1_b( \RR^+, \rho^2 d\rho; L^2(S, d\omega)) \cap
\rho^\mu L^2( \RR^+, \rho^2 d\rho; \widehat{\calD}_\psi(\widehat{\Dirac}))\}.
\]
This is complete with respect to the graph norm
\[
||u||^2_{Gr} := \int_{C(S)}  \left(  |\rho^{-\mu}u|^2 + |\rho^{-\mu} \rho\del_\rho u|^2 + 
|\rho^{-\mu}  \widehat{\Dirac} u|^2 \right)\, \rho^2 d\rho d\omega.
\]
This definition is local near each vertex, but can obviously be combined with $H^1_e \cap \calD_{\max}$ functions away from the vertices
which satisfy the boundary equation. Thus a section which lies in this domain satisfies the $\calT_\psi$ boundary conditions along 
the interior of each edge, and decays at the rate $\rho^{\mu-3/2}$ at each vertex, cf.\ Proposition \ref{4.19} below. 
Observe that we have defined these spaces carefully so that
\[
\Dirac:  \calD_{\psi, \mu} \longrightarrow \rho^{\mu-1} L^2( C(S), \rho^2 d\rho d\omega)
\]
is bounded.    One very important aspect is that the restriction of $\Dirac$ to a neighborhood of any vertex takes
the form \eqref{Dircps}, but where the cross-sectional operator $\widehat{\Dirac}$ is now self-adjoint given this 
choice of boundary conditions.  In particular, this realization of $\widehat{\Dirac}$ has discrete spectrum. 

\medskip

\noindent{\bf Indicial roots:}  
Exactly as in the previous settings, we can define the indicial roots $\gamma$ of $\Dirac$ relative to this choice of 
boundary conditions on the edges, as follows.   A number $\gamma$ is an indicial root of $\Dirac$ at $p$ if there exists 
a section $v(\omega) \in \widehat{\calD}_\psi$ on the cross-section $S_p$ such that
\[
\Dirac ( \rho^\gamma v(\omega)) = 0.
\]
(This is an exact solution near $p$ because we are in the model case.) Using \eqref{Diracsphcoor}, we calculate that for any $\lambda$ and $v$, 
\[
-\mathtt{cl}(\del_\rho) \Dirac( \rho^\gamma v(\omega)) = \rho^{\gamma-1} ( \widehat{\Dirac}  - (\gamma+1)) v(\omega),
\]
so $\gamma$ is an indicial root if and only if $\lambda  = \gamma + 1$ is an eigenvalue of $\widehat{\Dirac}$.  

Note in particular that if the nullspace of $\widehat{\Dirac}$ is nonempty, then $\gamma = -1$ is an indicial root of $\Dirac$.  Furthermore,
since the spectrum $\{\lambda_j\}$ of $\widehat{\Dirac}$ is reflection-invariant across $0$, the indicial roots
of $\Dirac$ at $p$ are symmetric around $-1$, i.e., of the form
\[
\gamma_j^\pm = -1 \pm \lambda_j, \qquad \lambda_j \in \mathrm{spec}\,(\widehat{\Dirac}) \cap [0,\infty).
\]


\medskip

\noindent{\bf Regularity and expansions:}  
We next review some facts about higher regularity.
\begin{proposition}
Suppose that $u \in \calD_{\psi,\mu}$ and $\Dirac u = f \in \calC^\infty_0$ is supported 
away from $\rho = 0$.  Then 
\[
u \sim \sum_{\gamma_i > \mu-3/2} A_i \phi_i(\omega) \rho^{\gamma_i} = \sum_{\lambda_i > \mu-1/2}  A_i \phi_i(\omega) \rho^{\lambda_i-1}
\]
as $\rho \searrow 0$, where the $\phi_i$ are the eigensections of $(\widehat{\Dirac}, \widehat{\calD}_\psi)$ with eigenvalue $\lambda_i$ 
and the $A_i$ are constants.   Furthermore, $u$ is polyhomogeneous along the edges. 
\label{4.19}
\end{proposition}
\begin{proof}
The proof of this is essentially exactly the same as in the case of isolated conical singularities.  Multiplying by a cutoff function,
we can assume that both $u$ and $f$ are supported in $\{\rho \leq 1\}$.   Now take the Mellin transform in $\rho$,
\[
u \longmapsto u_M(\xi, \omega) = \int_0^\infty \rho^{-i\xi - 1} u(\rho,\omega)\, d\rho.
\]
This is simply the Fourier transform in the variable $t = \log \rho$.  By the support properties of $u$ we can extend $u_M$ holomorphically
to $\zeta \in \CC$ with $\zeta = \xi + i\eta$ satisfying $\eta > 0$, taking values in $L^2(S)$. 

Taking the Mellin transform of the equation $\rho \Dirac u = \rho f$, we see that
\[
(\widehat{\Dirac} + i\zeta + 1) u_M = f_M.
\]
Furthermore, since $f$ is supported away from $\rho = 0$, $f_M$ is entire in $\zeta$. If $f$ is supported away from the edges,
or even just polyhomogeneous along these edges, then $f_M$ has the same property.  Furthermore,
\[
u_M = (\widehat{\Dirac} + i\zeta + 1)^{-1} f_M.
\]

We can recover $u$ from $u_M$ by taking the inverse Mellin transform, $u(\rho,\omega) = (2\pi)^{-1} \int u_M(\zeta, \omega) \rho^{i\zeta}\, d\xi$.
At first, the integration is taken along $\mathrm{Im}\zeta = 0$, but we can shift the contour down to any line $\eta = c$, using that 
$(\widehat{\Dirac} + i\zeta)^{-1}$ is meromorphic in $\zeta$, so long as this inverse does not have a pole along that line.  However, these poles
are at the values $\lambda_j$, where $\lambda_j$ is an eigenvalue of $\widehat{\Dirac}$. (Note that we know beforehand that $u_M$ has
no poles in the upper half-plane.)  Shifting the contour down, across each such pole, produces another term $\rho^{\lambda_j - 1}$ 
in the expansion of $u$.

The polyhomogeneity of $u$ along the edges is proved using finer properties of the inverse $G$ to $\Dirac$ on this cone. This
argument can be carried out precisely as in \cite{MazzeoWitten2}.
\end{proof}
The usual caveat applies that this expansion is not necessarily convergent.    An elaboration of this argument also proves the
\begin{proposition}
If $\Dirac u = f$ and $f$ polyhomogeneous on the space obtained by blowing up first the vertex and then the singular edges,
then $u$ is also polyhomogeneous on this space.
\end{proposition}

Finally, there is a partial regularity statement:
\begin{proposition}
Suppose that $u \in \calD_{\psi,\mu}$ for some $\mu > 0$ and $\Dirac u = f \in L^2$.  Then
\[
u = \sum A_i \phi_i(\omega) \rho^{\gamma_i} + \tilde{u} 
\]
where the $\gamma_i$ lie in the set of indicial roots in the interval $(-3/2, \mu-1/2)$, and $\tilde{u} \in \calD_{\psi,\mu+1}$.
\end{proposition}

\subsection{Global theory}
We now state the relevant global theory for $\Dirac$ on $M \setminus \Sigma$.  This is based on the principle that
local parametrices for $\Dirac$ constructed near every point of $M$ can be patched together to give a global parametrix
with compact remainders.   In fact, it suffices to patch together a parametrix in the edge calculus defined outside
the union of some small balls around each vertex with local parametrices in each of these balls. 

Enumerating the vertices as $\{q_1, \ldots, q_N\}$, and choosing $\varepsilon$ sufficiently small, let $\calU$ 
denote the union of balls $B_\varepsilon(q_i)$ and $\calV$ the complement of the union of the closed balls 
$\overline{B}_{\varepsilon/2}(q_i)$.    The parametrix $G_\calU$ can be constructed exactly as in Section 3.
It satisfies $I - \Dirac \circ G_\calU = R_{\calU, 1}$, $I - G_\calU \circ \Dirac = R_{\calU, 2}$, where 
the error terms $R_{\calU, i}$ can be chosen to map into smooth sections over $\calU$ which vanish
to infinite order along $\Si \cap \calU$.   The range of $G_\calU$ is precisely the domain $\calD_\psi$ over $\calU$.

Next, assume first that the metric $g$ is Euclidean near $q_j$ and $\Si$ is a union of straight rays in this ball.
Given a nonindicial weight $\mu$ at $q_j$, we constructed in the proof of Proposition \ref{4.19} an operator 
\[
G_{q_i}: \rho^{\mu-1}L^2( B_{\varepsilon}(q_i)) \longrightarrow \calD_{\psi,\mu}(B_\varepsilon(q_i)) \subset \rho^\mu L^2(B_\varepsilon(q_i))
\]
which satisfies 
\[
I - \Dirac \circ G_{q_i} = 0, \quad I - G_{q_i} \circ \Dirac = 0.
\]
The error terms vanish identically in this model case.  Let $G_{\calV}$ denote the `union' of these local inverses.

In the general case, 
\[
\Dirac = \mathtt{cl}(\del_\rho) \left( (\frac{\del\,}{\del \rho} + \frac{1}{\rho}) - \frac{1}{\rho} \widehat{\Dirac}(\rho)\right) + E;
\]
here $E$ is an error term that decays relative to the leading part (the expression in parentheses on the right).  The small complication
here is that $\widehat{\Dirac}(\rho)$ is a smoothly varying family of Dirac operators on $S^2 \setminus \{p_1, \ldots, p_{2k}\}$ and
we must consider its action on a $\rho$-dependent domain $\widehat{\calD}_{\psi(\rho)}$.  The limit as $\rho \to 0$
is simply $\widehat{\Dirac}(0)$ on $\widehat{\calD}_{\psi(0)}$.    The construction of a parametrix with good remainder terms 
is a technical exercise which is straightforward but lengthy.  

Now choose a partition of unity $\{\chi_U, \chi_V\}$ relative to the open cover $\calU \cup \calV$, and choose
$\tilde{\chi}_\calU \in \calC^\infty_0(\calU)$ to equal $1$ on the support of $\chi_\calU$, and similarly,
$\tilde{\chi}_\calV \in \calC^\infty_0(\calV)$ to equal $1$ on the support of $\chi_\calV$. Finally, define
\[
G = \tilde{\chi}_\calU G_\calU \chi_\calU + \tilde{\chi}_\calV G_\calV \chi_\calV.
\]
This operator maps $\rho^{\mu-1}$ into $\calD_{\psi,\mu}$. A straightforward calculation yields that both
\[
I - \Dirac \circ G\ \ \mbox{and}\ \ \ I - G \circ \Dirac
\]
are compact operators.

We have now proved the central
\begin{theorem} Let $\mu$ be any weight parameter which is not indicial at any one of the vertices $q_i$. Then
\[
\Dirac: \calD_{\psi, \mu} \longrightarrow \rho^{\mu-1}L^2(M \setminus \Si; \slS \otimes \calI)
\]
is Fredholm.
\label{mFt}
\end{theorem}

\subsection{Calculation of the index}
We come at last to our main theorem, namely a formula for the index of $(\Dirac, \calD_{\psi, \mu})$ when 
$\mu > 3/2$.   The weight $3/2$ is critical for the geometric interpretation since $\rho^\gamma \in \rho^\mu L^2$ 
implies $\gamma > \mu - 3/2$, so if $\mu > 3/2$ then we are restricting to sections which decay at each
vertex, which is natural for the deformation problem since it corresponds to variations which fix the vertices.
 The restriction $\mu > 5/2$ is also natural since it corresponds to variations which fix the incoming tangents at that vertex.

The proof has two steps. We first have to calculate the index of $\Dirac: \calD_{\psi,\mu} \to \rho^{\mu-1}L^2$ at some
special value of $\mu$ and then use an analogue of the relative index theorem to calculate the change of index as $\mu$
crosses other indicial weights.  As before, denote by $\Ind(\mu)$ the index of $\Dirac: \calD_{\psi,\mu} \to \rho^{\mu-1}L^2$.

Just as for the punctured sphere, the formal adjoint of $(\Dirac, \calD_{\psi, \mu})$ is $(\Dirac, \calD_{\psi, 1-\mu})$. 
Indeed, if $u \in \calD_{\psi,\mu}$, then $\Dirac u \in \rho^{\mu-1}L^2$, so the pairing $\langle \Dirac u, v \rangle$ is well-defined 
when $v \in \rho^{1-\mu}L^2$ and the boundary terms at the vertices all vanish. Their vanishing implies $\langle \Dirac u, v \rangle = \langle u, \Dirac v\rangle$.
We conclude from this that if $\mu = 1/2$ is not an indicial weight, then $(\Dirac, \calD_{\psi, 1/2})$ is self-adjoint, and hence 
$\Ind(1/2) = 0$.

If $1/2$ is an indicial weight we must argue differently.  We argue as in Corollary \ref{jumpind}.  First note that the relative index 
theorem \cite[Theorem 6.5]{Melrose93_AtiyahPatodiSingerBook} still holds in this setting. Indeed, the proof of that theorem
involves integrations by parts on the cross-section, and these are justified by our choice of boundary conditions. The argument 
then proceeds exactly as if the cross-section were a closed manifold. This gives that $\Ind(1/2 + \varepsilon) - \Ind(1/2-\varepsilon)$ 
is the negative of the sum over all vertices $p$ of the multiplicity of the indicial weight $1/2$.  Since the problem is symmetric at weight $1/2$
(or more accurately, the dual of $(\Dirac, \calD_{\psi,\mu})$ is $(\Dirac, \calD_{\psi,1-\mu})$), we also have that 
$\Ind(1/2-\varepsilon) = - \Ind(1/2 + \varepsilon)$. Thus arguing exactly as in  Corollary \ref{jumpind}, we see that this
jump in index across this particular indicial weight is an even number. In that setting that value equaled $2k$, the number of 
punctures on the sphere, but here it is just some global invariant, which we write as $-2H$.  Thus
\[
\Ind(1/2 + \varepsilon) = -H.
\]

Finally, let $\mu > 1/2$ be any nonindicial weight.  Then
\[
\Ind(\mu) - (-H) = \Ind(\mu) + H = \sum_{i=1}^N\sum_{\gamma^i_j \in (0, \mu-1/2) \cap \Lambda_i} M^i_{\gamma^i_j},
\]
where the sum is over all indicial roots $\gamma^i_j \in \Lambda_i$ over all vertices $q_i$, where $0 < \gamma^i_j < \mu-1/2$, and 
$M^i_{\gamma^i_j}$ is algebraic multiplicity of that indicial root. 

\section{Parametrix constructions} 
We conclude this paper by briefly sketching the parametrix constructions used throughout this paper. Some of this was
discussed already in Section 3.5 when $\Si$ is smooth, but it may be helpful to contrast the various types of operators we have
encountered all in one place. 

There are three different levels of degeneracy: the conic singularity exhibited by the Dirac operator on the punctured $2$-sphere, 
the simple edge singularity on $M^3$ when $\Si$ is a smooth embedded curve, and the iterated edge singularity when $\Si$ is 
an embedded graph.  The parametrices constructed in each of these three cases lie in what are called the $b$-, edge and iterated
edge calculi. Each such `calculus' is a collection of pseudodifferential operators which is almost closed under composition (if it
were closed under composition we would call it an algebra rather than a calculus), with degeneracies modelled on the singular behavior
of the Dirac operator in each of these cases.  

The goal in each setting is to construct an approximate inverse to the Dirac operator, and from this a generalized inverse.
To do so, we construct a new adapted class of pseudodifferential operators, the elements of which have the same type 
of scaling and degeneracy properties as the singular operator in question.  Each of these three cases relies on an
ansatz about the putative geometric nature of the Schwartz kernels of these operators. The calculus is then determined
by fixing this `geometric type' for the Schwartz kernels of constituent elements.  One must show that the corresponding
class of operators is closed under composition (when it is defined) and has reasonable mapping properties between
weighted Sobolev (and H\"older) spaces.  There is also a multi-step symbol calculus which is used in the construction of
parametrices for the elliptic operators in this setting, and in particular the corresponding Dirac operator.  If this 
can all be done, the ansatz will have been vindicated. 

The $b$-calculus was historically the first to be developed, and this was done by Melrose; treatments can be found in
\cite{Melrose93_AtiyahPatodiSingerBook} and \cite{Mazzeo91_EllTheoryOfDiffEdgeOp}.   The edge calculus was constructed soon after, see
\cite{Mazzeo91_EllTheoryOfDiffEdgeOp}, with further developments in \cite{MaVe} and \cite{Usula}, \cite{Usula2}. The paper 
\cite{MazzeoWitten1} contains an expository section about this theory.  The iterated edge calculus has been partially 
developed in several settings, see \cite{MazzeoMontcouquiol11_Stoker}, and \cite{MazzeoWitten2}. A more systematic
and comprehensive treatment is forthcoming. We provide here a summary overview of what is involved in each case.  
We present these three settings in sequence to illustrate how the increasing complexity of degeneracies must  be handled. 

The first step is to pass from the singular space to a manifold with boundary or corners. This is done using real blow-ups,
as described and used extensively above. Thus the punctured sphere $S^2 \setminus \{p_1, \ldots, p_N\}$ is replaced
by the manifold with boundary obtained by blowing up each $p_j$.  For the $3$-submanifold $M$ with embedded curve $\Si$
we blow up $\Si$, resulting in a manifold with boundary $M_\Si$, where the boundary is the total space of a circle bundle over $\Si$.
Finally, if $\Si$ is an embedded graph, we first blow up the vertices of the graph, and then the edges. The resulting space $M_\Si$
is a manifold with corners, with two types of boundary faces, one covering the vertices and the other the edges.  For simplicity
here we refer to the blown up space in each instance as $X$. 

A generalized inverse is an operator $G$ satisfying 
\[
\Dirac \circ G = \mathrm{Id} - \Pi_2, \ \ G \circ \Dirac = \mathrm{Id} - \Pi_1,
\]
where $\Pi_1$ and $\Pi_2$ are the orthogonal projectors onto the cokernel and the kernel of $\Dirac$.   We have omitted any mention
of boundary conditions; these are usually encoded in the choice of domain $\calD$ in a given weighted $L^2$ space on which 
$\Dirac$ acts.  Boundary conditions are not needed if there exists a weight $\mu$ such that 
\begin{equation}
\Dirac: \rho^\mu H^2_\sharp \to \rho^{\mu-1} L^2
\label{DmF}
\end{equation}
is Fredholm. Here $H^1_\sharp $ denotes either $H^1_b$, $H^1_e$ or $H^1_{\ie}$, the Sobolev spaces specifically
adapted to those particular degeneracies.  Of course, whether \eqref{DmF} is Fredholm or not is not known beforehand,
and is only determined if one can construct a parametrix with compact remainders. 

A parametrix for $\Dirac$ is a pseudodifferential operator $B \in \Psi^{-1}_\sharp(M; \sS\otimes \calI)$ which defines 
a bounded operator $B: r^{\mu-1}L^2 \to r^\mu H^1_\sharp$ and satisfies 
\[
\Dirac \circ B = I - R_2,\ \ G \circ \Dirac = I - R_1,
\]
where the $R_j \in  r^\epsilon \Psi^{-\infty}_\sharp$ are smoothing operators which map $r^{\mu-1}L^2 \to r^{\mu-1+\epsilon}H^\ell_\sharp$
and $r^\mu L^2 \to r^{\mu+\epsilon}H^\ell_\sharp$, respectively, for all $\ell \geq 0$ and some $\epsilon > 0$.  The $L^2$ Arzela-Ascoli
theorem shows that these error terms are compact. Applying standard functional analysis arguments, we deduce that 
$\Dirac$ is Fredholm, and hence for abstract reasons has a generalized inverse $G$, i.e., an operator which is an inverse
from the range to the orthogonal complement of the nullspace. A standard argument, see \cite {Mazzeo91_EllTheoryOfDiffEdgeOp}[Section 6],
shows that this generalized inverse $G$ is an element of $\Psi^{-1}_\sharp$.

The next simplest situation, which is directly relevant to the problem of $\Z_2$ harmonic spinors, is when \eqref{DmF} is semi-Fredholm, i.e.,
where this map has closed range for some particular $\mu$ and either a finite dimensional kernel or cokernel, but not both.  Suppose that
it has a finite dimensional nullspace.  We can reduce to the Fredholm case by considering the second order operator $\Dirac^* \Dirac$.  
A basic example to keep in mind is when $\mu=1$ and $\Si$ is a smooth curve. Then $\calD^*$ is identified with
$\Dirac: L^2 \to r^{-1}H^{-1}_e$, so $\calD^* \calD: r H^1_e \to r^{-1}H^{-1}_e$, and this latter map is Fredholm.
If $\mathcal G$ is a Fredholm parametrix for this second order operator, then $\mathcal G \circ \Dirac^*$ is a left parametrix
for $\Dirac$, which implies that $\Dirac$ has finite dimensional nullspace and closed range.   

We defer discussion of the case where boundary conditions are needed since this requires a slightly more elaborate setup.

We next outline the steps needed to construct a parametrix in the Fredholm case.   The main idea is that, by the Schwartz kernel theorem, 
the parametrix $B$ is a distribution on $X \times X$, but must exhibit features near $\del X$ which reflect the degeneracy of $\Dirac$.
These features are encoded in the assertion that $B$ is polyhomogeneous on a certain blowup of $X^2$.  The blowups needed
in each of the three cases are slightly different:
\begin{itemize}
\item In the case of isolated singularities, we replace $X^2$ by the $b$ double space $X^2_b$. This is defined by blowing up the corner 
$(\del X)^2$ in $X^2$. Actually, in our case, $\del X$ has multiple components, and we only blow up the components of $(\del X)^2$ 
which intersect the diagonal of $X^2$.  Using local coordinates $(r,\theta)$ near a given boundary component, then $(r,\theta, \tilde{r},
\tilde{\theta})$ are local coordinates near the associated corner, and we are blowing up the entire codimension two set $\{r = \tilde{r} = 0\}$.
This can be seen as introducing polar coordinates $R, \omega$ where $R\cos \omega = r$, $R \sin \omega = \tilde{r}$, and attaching
the entire $R = 0$ interval $0 \leq \omega \leq \pi/2$ for each $y, \tilde{y}$.  The space $X^2_b$ has a new front face $\fff$.
\item  When $\Si$ is a smooth curve, we replace $X^2$ by the edge double space $X^2_e$, which is defined by blowing up the
fiber diagonal of $(\del X)^2$. In local cylindrical coordinates $(r,\theta,y)$ as used earlier in this paper, we have coordinates
$(r,\theta,y,\tilde{r},\tilde{\theta}, \tilde{y})$ on $X^2$ near the corner $r = \tilde{r} = 0$, and we blow up the
submanifold $\{r = \tilde{r} = 0, y = \tilde{y})$. This corresponds to introducing local spherical coordinates $R$ and $\omega \in S^2_{++}$,
$R \omega = (r, \tilde{r}, y - \tilde{y})$. Here $S^2_{++}$ is the spherical orthant $\{\omega = (\omega_0, \omega_1, \omega_2: \omega_0, \omega_1 \geq 0\}$.
The front face $R=0$ created in this blowup is now fibered by copies of $S^2_{++} \times S^1_\theta$, with base $\Si\times S^1$. This can be seen using the local
coordinates $(R, \omega, \theta, \tilde{\theta}, \tilde{y})$. 
\item Finally, when $\Si$ is a graph, then $X$ has two types of boundary faces, those covering the edges of $\Si$ and those covering the vertices.
We can introduce local coordinates $\rho, r, \theta$ near a corner where these faces intersect, see Section 2.2. Thus $\rho$ is comparable
to the variable $y$ along the edges, and $(r,\theta)$ are coordinates on the punctured sphere (blown up) around a particular puncture.   
In terms of coordinates $(\rho, r, \theta, \tilde{\rho}, \tilde{r}, \tilde{\theta})$ on $X^2$, the iterated edge double space $X^2_{\ie}$ is defined by 
first blowing up the corner $\{\rho = \tilde{\rho} = 0\}$, and then the submanifold $\{ r = \tilde{r} = 0, \rho = \tilde{\rho}\}$. 
In other words, we first do a $b$ blowup for the faces covering the vertices of $\Si$ and then an edge blowup for the faces covering the edges.
This space has two different front faces, which we write as $\fff_\rho$ and $\fff_r$, respectively.
\end{itemize}

The parametrix construction in each of these cases proceeds, as in so many other arguments in microlocal analysis, by finding successively
good approximations to the object one wants. These approximations are chosen to make the error terms $R_k^{(1)} =  I - B_k \Delta$,
$R_k^{(2)} = I - \Dirac B_k$ successively smoother and `smaller'.   Here, to be very specific to our problem, we are writing 
$\Delta = \Dirac^* \Dirac$. There are two crucial steps (as well as some later ones that
are less essential, but yield even nicer error terms).  The first uses the classical symbol calculus along the lift of the diagonal
to $X^2_\sharp$ to ensure that the error terms are operators of order $-\infty$. In other words, the initial contribution to the
parametrix is chosen so that the error term has no singularity along the diagonal.  For nondegenerate elliptic operators, this
is the standard elliptic parametrix construction. Here, of course, $\Dirac$ appears to become singular at $\del X$, so this does
not seem so practicable. However, the blowups $X^2_\sharp$ have been chosen precisely so that the lift of the Dirac operator
to this space, in each of the three cases, is uniformly elliptic along the entire lifted diagonal, at least up to an overall singular
factor of the boundary defining function.  This means that the classical elliptic parametrix construction may be carried
out here as well. 

The reason why the process needs to continue is that the error terms $R_0^{(j)}$ obtained in this first stage
are smoothing but do not provide any extra decay beyond what is `expected'. Thus $R_0^{(j)}: r^{\mu-2}L^2 \to r^{\mu}H^\ell_\sharp$.
We need to find a correction $B_1$ to the parametrix $B_0$ so that the corresponding remainders $R_1^{(j)}$ are compact
operators, which would be the case if $R_1^{(j)}: r^{\mu-2} L^2 \to r^{\mu+1}H^\ell_\sharp$ (any improvement, e.g. to weight $r^{\mu+\epsilon}$
would be sufficient).  This in turn can be shown to hold if the restriction of the $R_1^{(j)}$ to the front face(s) of $X^2_\sharp$ vanish.

All of this suggests that we need to consider the restrictions of the equations $\Delta B_1 = B_0$, $B_1 \Delta = B_0$ to these front
faces.  It is an elementary but enormously important aspect of the theory that this restricted problem, which is
called the `normal equation' has a simple characterization:
\begin{itemize}
\item On the front face of $X^2_b$, this normal equation takes the form $\Delta_{\RR^+ \times S^1} u = f$ for some smooth compactly 
supported $f$. (More accurately, one must solve a problem of this form for each $\tilde{\theta} \in S^1$.)  Here $f$ is the restriction
of $B_0$ and $u$ the restriction of $B_1$ to the front face.   This equation can be simplified by taking the Mellin transform
in $\RR^+$, to reduce to a family of Laplace-type equations on $S^1$ with the Mellin transform variable $\zeta$ as a parameter.
This equation turns out to be solvable in the appropriate weighted space if and only if the real part of $\zeta$ is not equal
to one of the indicial roots $\ell + \frac12$, $\ell \in \ZZ$. 
\item For the simple edge setting, the normal equation is identified with the Laplace equation on $\RR^2_+ \times S^1$. Here this upper half-plane
should be regarded as identified with the quarter-sphere $S^2_{++}$ via stereographic projection.  This is a slightly more complicated
equation to solve, but turns out to be tractable on our given weighted space, once again so long as the weight is non-indicial.
The method is to take the Fourier transform in the tangential variable in $\RR^2_+$, and then perform a certain rescaling of the
other upper half-space variable, which leads to a `Bessel-type'  Laplace equation on $\RR^+ \times S^1$.  This too is solvable,
and the solutions are the restrictions of the putative correction term $B_1$ to the parametrix.
\item In the iterated edge case, one needs to solve two separate model problems, the first being a $b$-type normal equation on
$\fff_\rho$ and the second an edge-type normal equation on $\ff_r$. 
\end{itemize}
Having carried out these steps, one may then define $B_1$ to be a smooth extension of the solution to the relevant one of these
equations. Then, by construction, the error terms coming from the parametrix $B_0 + B_1$ are compact.

In \cite{Mazzeo91_EllTheoryOfDiffEdgeOp} it is shown how to refine this construction further in the $b$- and edge settings
to get remainder terms which are even smaller. This is also possible in the iterated edge setting, and details will appear in
a forthcoming paper.  The result is a parametrix for which the remainder terms vanish to infinite order  at the front face(s),
which makes certain further arguments more straightforward.

Finally, we describe briefly how to modify these arguments when boundary conditions are introduced. We shall be once again
specific to the problem of interest in this paper.  As noted above, $\Dirac$ is already Fredholm in the $b$ setting, so
let us focus on the two maps 
\[
\Dirac: r H^1_e \longrightarrow L^2, \quad \mbox{and}\qquad \Dirac: H^1_e \longrightarrow r^{-1}L^2,
\]
when $\Si$ is smooth. As we have stated carefully earlier in the paper and have outlined the proofs of above, these maps
are both semi-Fredholm, but the first has finite dimensional nullspace and infinite dimensional cokernel and the second
has infinite dimensional nullspace and finite dimensional cokernel. 

Suppose now that we have some $f \in L^2$ and we wish to solve $\Dirac u = f$.  Since the first map above has such
a large cokernel, it is unlikely that we can find a solution $u \in r H^1_e$.  However, if we embed $L^2 \hookrightarrow r^{-1}L^2$
and employ the essential surjectivity of the second map, then modulo some finite dimensional obstruction space, we
can solve $\Dirac u = f$ with $u \in L^2$.  The fact that $f \in L^2$ rather than $r^{-1}L^2$ means that the soluion
$u$ possesses some additional regularity, and has a partial asymptotic expansion as $r \to 0$, with leading pair
of coefficients $(a_0, b_0)$. A priori these are $H^{-1/2}$ functions along $\Si$.   The boundary operator $\calT_\psi$
we have defined is an algebraic condition on this pair of leading coefficients.  A more involved version of the 
parametrix construction, which is carried out in \cite{MaVe}, see also \cite{Usula}, \cite{Usula2}, provides a parametrix
$B: L^2 \to H^1_e$, with every element of the range of $B$ having a pair of leading coefficients which satisfy
this boundary condition.    This new parametrix is an adaptation of the method due originally to Calderon
and Boutet de Monvel to this edge setting.    The upshot is that we obtain a parametrix for $\Delta$ which
respects the boundary conditions and which leaves compact remainder terms.

We come at least to the final situation, where $\Si$ is a graph, and we are trying to construct the parametrix
near the vertices of $\Si$.  We can impose the boundary operator along each edge, as above.  The key observation
is that after a diffeomorphism which straightens out each edge near a vertex, we can think of
the local picture of $M \setminus \Si$ as equivalent to a neighborhood of $0$ in the cone over the punctured sphere.
On the blown up spaces, a neighborhood of some boundary face covering a vertex of $\Si$ is identified with
$[0,1)$ times the punctured sphere blown up at the punctures.  The induced operator is then well approximated
by the Diract operator on a cone over the punctured sphere.   Thus near these faces we may apply the $b$-version
of the parametrix construction above, but remembering throughout that the cross-section is no longer 
a compact manifold, but instead a punctured sphere, or surface with boundary. The Dirac operator in this
identification is just a conic operator, but where the cross-sectional operator itself is a conic Dirac operator on
the punctured sphere. As such, it is natural to identify weighted spaces $\rho^\mu L^2$.  The only remaining
big difference is that the indicial roots in this conic identification are no longer something simple like half-integers,
but are now determined by the spectrum of the induced Dirac operator (or its Laplacian) on the punctured sphere. 

This whole discussion boils down a lot of technical information spread out over several papers, but is intended
to give the reader a guide to the techniques so that the analytic conclusions which we have used throughout
the paper do not appear like ``dei ex machina''.

\section{The linearized deformation problem and its index when $\Si$ is a smooth curve}
In this final section, we connect the problem studied in this paper with its original motivation:  the Fredholm deformation 
theory of $\ZZ_2$-harmonic spinors. Our goal here is to explain why the type of boundary condition introduced and
studied in this paper is relevant to this deformation problem.   We restrict attention in this section to the case
where $\Si$ is a smooth curve without vertices.  A careful analysis of the deformation problem in the presence
of vertices involves some new features, and we shall return to this elsewhere.

By definition, an `admissible triple' $(g_0, \Sigma_0, \psi_0)$ consists of a metric $g_0$, an element $\psi_0 \in Z_1$, i.e., 
a bounded $\ZZ_2$-harmonic spinor, and its branching set $\Sigma_0$.  As first carried out in \cite{Takahashi15_Z2HarmSpinors_Arx}, 
but see also \cite{Parker-DefHS3M} for an alternate comprehensive treatment, we seek to understand the set of all nearby admissible 
triples $(g,\Sigma,\psi)$.  At the infinitesimal level, this amounts to characterizing the infinitesimal deformations $(\dot{g}, \dot{\Sigma}, 
\dot{\psi})$ of this problem, and this is what we describe here.  The more difficult analysis to `integrate' these infinitesimal deformations 
to one-parameter families $(g_t, \Sigma_t, \psi_t)$ of admissible triples is the topic of the two papers cited above. 

We consider only a restricted version of this problem where the branching set $\Si_0$ is allowed to vary, but not the metric. 
This statement must be interpreted properly.  The space of embedded smooth curves in $M$ is a Frechet manifold with tangent 
space at $\Si_0$ equal to the space of smooth sections of the normal bundle $N \Si_0$. A neighborhood $\calU$ of the zero 
section in $\calC^\infty(\Si_0, N\Si_0)$ parametrizes a neighborhood of $\Si_0$ in this space of smooth curves. 
An infinitesimal variation of the branching set is a particular section.  We may extend this section to a smooth vector field on $M$, 
and then take the corresponding one-parameter family of diffeomorphisms $F_\epsilon$.  It is much simpler in the analysis to consider 
this problem with a fixed branching set, thus we pull back the metric and equation with respect to $F_\epsilon$, and thus consider 
variations where $\Si_0$ is fixed, but where the metric $g$ varies by a diffeomorphism. The associated infinitesimal variation
$\dot{g}$ is then of a very special type, which is used extensively in the calculations below. Thus the more accurate
statement is that we consider this deformation problem with fixed branching set, but with metrics varying by diffeomorphisms only. 

The material in this section is meant to be primarily motivational, so we do not specify a Banach or Hilbert regularity on these
sections or diffeomorphisms.  This is, in fact, a  delicate issue: there are significant technical difficulties in carrying out
the  nonlinear deformation theory because of an underlying loss of regularity, and ultimately one is forced to work in the Frechet setting.

Given such a normal section $\zeta$, denote by $\Si_\zeta$ the curve $\{\exp_p( \zeta(p)): p \in \Si_0\}$.   It is more convenient to 
consider variations of admissible triples where the variation of the branching set is transferred to a variation of the metric, so to 
this end, choose a smooth map from the space $\calU$ of small normal sections to the space of diffeomorphisms $\Psi_\zeta$ of $M$ 
close to the identity which has the property that $\Psi_\zeta(\Si_0) = \Si_\zeta$ and $\Psi_\zeta$ equals the identity outside of a 
fixed small tubular neighborhood of $\Si_0$. By conjugating with these diffeomorphisms, we obtain a smooth family of Dirac 
operators $\Dirac_\zeta$ acting on spinors branching over the same fixed curve $\Si_0$. These operators $\Dirac_\zeta$ are defined 
with respect to the pulled back metric $\Psi_\zeta^* g_0$, the pulled back spin structure, and the pullback of the metric on the 
$\RR$-bundle $\calI$ with monodromy $\ZZ_2$ around $\Si_0$, and the infinitesimal variation $\dot{\Dirac}$ at $\zeta = 0$
must be computed taking into account each of these dependencies. 

Define the map
\[
\calF: H^2(\Si_0) \oplus r H^1_e( M \setminus \Si_0; \sS \otimes \calI) \longrightarrow L^2(M \setminus \Si_0; \sS \otimes \calI),
\]
by 
\[
\calF( \zeta, \phi) = \Dirac_\zeta \phi.
\]
The full deformation problem involves determining the solution space to the nonlinear equation $\calF(\zeta, \phi) = 0$, but here
we only need to study its linearization.  The computation of this linearization is complicated by all of the implicit identifications 
mentioned above, but is explained carefully by Parker \cite[Section 5.2]{Parker-DefHS3M}; we refer to that paper for details. The 
key ingredient is a well-known formula due to Bourguignon and Gauduchon for the variation of the `untwisted' Dirac operator.  
Write 
\[
\dot{g}_\zeta = \left. \frac{d\,}{d\epsilon} \Psi_{\epsilon \zeta}^* g \, \right|_{\epsilon =0},
\]
and let $\{e_1, e_2, e_3\}$ be an oriented orthonormal frame for $g_0$ with $e_3$ tangent to $\Si_0$. Then \cite[Corollary 5.9]{Parker-DefHS3M} 
states that
\begin{equation}
D \calF |_{(0, \psi_0)} ( \zeta, \phi) = \frac12 \left( - \sum_{i, j} \dot{g}_\zeta(e_i, e_j) e^i \cdot \nabla _{e_j}
+ d\, \mathrm{tr} (\dot{g}_\zeta) + \mathrm{div} (\dot{g}_\zeta + Q(\zeta))\right) \psi_0 + \Dirac_0 \phi,
\label{varDirac}
\end{equation}
where $Q$ is a linear operator in $\zeta$ of order $0$, the covariant derivative, trace and divergence are all with respect to $g_0$ 
and ``$\cdot$'' in the first term in parentheses is the Clifford multiplication.  Since we are interested in computing the index of
$D\calF$, we are free to omit $Q$, which is of lower order.  Thus, denoting the term in parentheses (without $Q(\zeta)$)
by $\calP_0(\zeta)$, we write
\[
D\calF |_{(0,\psi_0)}(\zeta, \phi) = \calP_0(\zeta) + \Dirac_0 \phi.
\]

The precise analysis of the various terms in this expression near $\Si_0$ is carried out in Sections 5 and 6 of \cite{Parker-DefHS3M}.
Only a closer examination of this formula reveals that this linearization is actually second order in $\zeta$, which explains why we have 
used $H^2(\Si_0)$ as the first component of the domain of $\calF$.  It is more easily apparent that the only term in $\calP_0$
that blows up like $r^{-1/2}$ is the one that involves differentiating $\psi_0$, i.e., 
\begin{equation}
-\frac12 \sum_{i, j} \dot{g}_\zeta(e_i, e_j) e^i \cdot \nabla _{e_j}^{g_0} \psi_0.
\label{mainterm}
\end{equation}

The remainder of this argument is due to Cliff Taubes, and we are indebted to him for explaining it to us and to Greg Parker
for passing on to us his careful notes of the argument. 

We now introduce an auxiliary object, the `stress-energy' tensor $T$ associated to any $\Phi \in Z_0$ and $\psi_0 \in Z_1$,
the kernels of $\Dirac$ in $L^2$ and $rH^1_e$, respectively. Relative to any orthonormal frame $\{e_i\}$, this is given by 
\[
T_{ij} = -\frac12 \left( \langle \sum (e_i \nabla_{e_j} + e_j \nabla_{e_i}) \psi_0, \Phi \rangle  
- \sum (\langle e_i \psi_0, \nabla_{e_j} \Phi\rangle + \langle e_j \psi_0, \nabla_{e_i} \Phi\rangle) \right).
\]
Here $\langle\cdot,\cdot\rangle$ denotes the Hermitian scalar product so that $T_{ij}$ is a complex-valued tensor. 
This has two key properties:
\begin{lemma}
The tensor $T$ is divergence free, i.e., $\sum \nabla_{e_j} T_{ij} = 0$. In addition, for any $\Phi \in Z_0$ 
\[
\langle \calP_0(\zeta), \Phi\rangle = -\frac12 \int_M  \langle \dot{g}_\zeta, T\rangle
\]
\label{divfree}
\end{lemma}
\begin{proof}
Both identities rely solely on the fact that $\Dirac_0 \Phi = \Dirac_0 \psi_0 = 0$. 

Since $T$ is tensorial, to show it is divergence-free it suffices to compute at a point $p$ and with respect to a frame such that $\nabla_{e_i} e_j(p) = 0$ 
for all $i, j$. Thus, 
\begin{equation*}
\begin{split} 
-2 \sum_j \nabla_{e_j} T_{ij}  = \sum_j & \left(\langle (e_i \nabla_{e_j}^2 \psi_0,\Phi \rangle + \langle e_i \nabla_{e_j} \psi_0, \nabla_{e_j} \Phi \rangle + 
\langle e_j \nabla_{e_j} \nabla_{e_i} \psi_0, \Phi\rangle  + \langle e_j \nabla_{e_i} \psi_0, \nabla_{e_j} \Phi\rangle\right. \\  
& \left. - \langle e_i \nabla_{e_j} \psi_0, \nabla_{e_j}\Phi\rangle - \langle e_i \psi_0, \nabla_{e_j}^2 \Phi\rangle - 
\langle e_j \nabla_{e_j} \psi_0, \nabla_{e_i} \Phi\rangle - \langle e_j \psi_0, \nabla_{e_j} \nabla_{e_i} \Phi\rangle\right). 
\end{split}
\end{equation*}
Since $\sum \nabla_{e_j}^2 = -\Dirac^2$ and Clifford multiplication by $e_j$ is skew-Hermitian, the first, fourth, sixth and seventh terms vanish, while
the second and fifth terms cancel. Hence this entire expression reduces to
\[
\sum_j \left(\langle e_j \nabla_{e_j} \nabla_{e_i} \psi_0, \Phi\rangle - \langle e_j \psi_0, \nabla_{e_j} \nabla_{e_i} \Phi\rangle \right).
\]
Now, $\nabla_{e_j}\nabla_{e_i} = \nabla_{e_i} \nabla_{e_j} +  \mathrm{Riem}_{ji}$, and $e_j \nabla_{e_i} \nabla_{e_j}= \nabla_{e_i} e_j \nabla_{e_j}$ at $p$, so we are left with
\[
\sum_j \left(\langle e_j \mathrm{Riem}_{ji} \phi_0, \Phi\rangle + \langle \phi_0, e_j \mathrm{Riem}_{ji} \Phi\rangle \right).
\]
The first term in each summand can be rewritten as $-\langle \mathrm{Riem}_{ji} \psi_0, e_j \Phi\rangle = \langle \psi_0, \mathrm{Riem}_{ji} e_j \Phi\rangle$, which reduces us to understanding the commutator $[ \mathrm{Riem}_{ji}, e_j ]$. 
Recall, cf.~\cite{Frie00} that the action of the curvature operator on spinors is given by $\mathrm{Riem}_{ji} \Psi = R_{ji} e_j e_i \Psi$ where each $R_{ji}$
is a two-form. Therefore 
\[
\sum_j e_j \mathrm{Riem}_{ji} = \sum_j e_j R_{ji} e_j e_i = - \sum_j R_{ji} e_i  = - \sum_j R_{ji} e_j e_i e_j,
\]
so the final term in our computation vanishes, and hence $T_{ij}$ is divergence-free.

As for the second identity, we compute 
\begin{equation*}
-2 \langle \calP_0(\zeta), \Phi \rangle  = \int_M   \langle \dot{g}_{ij} e_i \nabla_{e_j}\psi_0, \Phi\rangle 
- \langle \nabla_{e_j} (\mathrm{tr}\, \dot g) \,  e_j \psi_0 , \Phi \rangle + \langle \nabla_{e_j} \dot{g}_{ij} \, e_j \psi_0, \Phi \rangle.
\end{equation*}
The middle term here is of the form $\langle df \cdot \psi_0, \Phi\rangle = \langle \Dirac( f\psi_0) - f \Dirac \psi_0, \Phi\rangle = \langle \Dirac(f\psi_0), \Phi\rangle$,
so integrating and moving $\Dirac$ to the other side makes this term vanish.    Integrating by parts in the third term, we see that the whole expression equals
\[
\int_M  \dot{g}_{ij} (\langle e_i \nabla_{e_j} \psi_0, \Phi \rangle -  \langle e_j \psi_0, \nabla_{e_i} \Phi\rangle),
\]
and since the coefficient $\dot{g}_{ij}$ is symmetric in $i$ and $j$, we may symmetrize the integrand and recognize all of this as $\int_M \dot{g}_{ij} T_{ij}$, as claimed.
\end{proof}

Now observe that modulo terms which decay at least to first order at $\Si$, 
\[
\dot{g} \equiv \nabla_{e_3} h, \quad h = \begin{pmatrix} 0 & 0 & \zeta_1 \\ 0 & 0 & \zeta_2 \\ \zeta_1 & \zeta_2 & 0 \end{pmatrix}. 
\]
In fact, since $\zeta$ can be assumed to be parallel to first order in normal directions at $\Si$, we can add in derivatives in this direction and
realize $\dot{g} \equiv \delta^* \zeta$, again modulo terms vanishing to first order, where $\delta^*$ is the dual of the divergence operator $\delta$.   Now let $M_\epsilon = M \setminus D_\epsilon$, where
$D_\epsilon$ is the neighborhood of radius $\epsilon$ around $\Si$, and write $\mu$ be the $1$-form obtained by contracting $\zeta$ into $T$,
i.e., $\mu_j = \sum_i \zeta_i T_{ij}$; thus $-\delta \mu = \langle \delta^*\zeta, T \rangle - \langle \mu, \delta T \rangle$. 
Discarding (for the moment) the terms that vanish to higher order along $\Si$ and using that $\delta T = 0$, 
\begin{equation*}
\int_M \langle \dot{g}, T\rangle =  \lim_{\epsilon \to 0} \int_{M_\epsilon}  \langle \delta^*\zeta , T \rangle = - \lim_{\epsilon \to 0} 
\int_{\del M_\epsilon}  *\mu = - \lim_{\epsilon \to 0} \int_{\del M_\epsilon} \mu_r\,  r dy d\theta,
\end{equation*}
where $y$ is the parameter along $\Si$. 

Writing $T_{ir} = T( e_i, \del_r)$, $i = 1, 2$, then $\mu_r = \zeta_1 T_{1r} + \zeta_2 T_{2r}$, which can be conveniently written in complex notation as
\[
\mu_r = \frac{1}{2}\Big((\zeta_1 + i \zeta_2) (T_{1r} - i T_{2r}) + (\zeta_1 - i \zeta_2) (T_{1r} + i T_{2r})\Big).
\]
The reader should beware that the quantities $T_{1r} \mp i T_{2r}$ are not  complex conjugates of one anther since the individual terms $T_{ir}$ themselves are not necessarily real. 

We next compute these various terms 
\[
T_{ir} =-\frac{1}{2}\Big( \langle \cl(e_i) \nabla_{\del_r} \psi_0, \Phi\rangle + \langle \cl(\del_r) \nabla_{e_i} \psi_0, \Phi\rangle - \langle \cl(e_i) \psi_0, \nabla_{\del_r} \Phi\rangle 
- \langle \cl(\del_r) \psi_0, \nabla_{e_i} \Phi\rangle\Big).
\]
Each of these terms has an expansion in powers of $r$, with leading term $r^{-1}$, which is the only one that contributes to the limit of the integral. There
is also an overall factor of $1/2$ arising from differentiating $r^{\pm 1/2}$ in the expansions of $\psi_0$ and $\Phi$ with respect to $r$. Finally, even though $e_1$ and $e_2$ are not a coordinate frame, we can write $\nabla_{ e_1 - i e_2} = 2 \nabla_{\del_z}$ up to terms which vanish to higher
order in $r$ (in other words, the linear coordinates $x_1 e_1 + x_2 e_2$ on the fibers of the normal bundle of $\Si$ are identified with $z = x_1 + i x_2$.
In a similar way, we write $\cl( e_1 - i e_2) = 2 \cl(\del_z)$, with similar notion replacing $e_1 + i e_2$ by $2 \del_{\bar{z}}$. 
In terms of all of this, we have that
\begin{equation*}
\begin{split}
& (T_{1r} - i T_{2r}) = -\langle \cl(\del_z) \nabla_{\del_r} \psi_0, \Phi\rangle -  \langle \cl(\del_r) \nabla_{\del_z} \psi_0, \Phi\rangle
+ \langle \cl(\del_z) \psi_0, \nabla_{\del_r} \Phi\rangle +  \langle \cl(\del_r) \psi_0, \nabla_{\del_{\bar{z}}} \Phi\rangle, \\ 
& (T_{1r} + i T_{2r})= -\langle \cl(\del_{\bar{z}}) \nabla_{\del_r} \psi_0, \Phi\rangle -  \langle \cl(\del_r) \nabla_{\del_{\bar z}} \psi_0, \Phi\rangle
+ \langle \cl(\del_{\bar{z}}) \psi_0, \nabla_{\del_r} \Phi\rangle +  \langle \cl(\del_r) \psi_0, \nabla_{\del_{z}} \Phi\rangle.
\end{split}
\end{equation*}

We now recall that
\begin{equation*}
\psi_0 \sim \begin{pmatrix} c_1 \sqrt{z} \\ d_1 \sqrt{\bar{z}} \end{pmatrix} = \begin{pmatrix} c_1 e^{i\theta/2} \\ d_1 e^{-i\theta/2} \end{pmatrix} r^{1/2},\qquad 
\Phi \sim \begin{pmatrix} a_0/\sqrt{z} \\ b_0/\sqrt{\bar{z}} \end{pmatrix} = \begin{pmatrix} a_0 e^{-i\theta/2} \\ b_0 \, e^{i\theta/2} \end{pmatrix} r^{-1/2}
\end{equation*}
and
\[
\cl(\del_r) = \begin{pmatrix} 0 & i e^{-i\theta} \\ i e^{i\theta} & 0 \end{pmatrix},\qquad \cl(\del_z) = \begin{pmatrix} 0 & 0 \\ i & 0 \end{pmatrix}, \qquad 
\cl(\del_{\bar{z}})= \begin{pmatrix} 0 & i \\ 0 & 0 \end{pmatrix},
\]
hence
\begin{equation*}
\begin{aligned}
\nabla_{\del_r} \psi_0 &\sim \frac12 \begin{pmatrix} c_1 e^{i\theta}/2 \\ d_1 e^{-i\theta/2} \end{pmatrix} r^{-1/2},\quad  
&\nabla_{\del_r} \Phi &\sim -\frac12 \begin{pmatrix} a_0 e^{-i\theta/2} \\ b_0 e^{i\theta/2} \end{pmatrix} r^{-3/2} , \\ 
\nabla_{\del_z} \psi_0 &\sim \frac12 \begin{pmatrix} c_1/\sqrt{z} \\ 0 \end{pmatrix},\   &\nabla_{\del_{\bar{z}}} \psi_0 &\sim 
\frac12 \begin{pmatrix} 0 \\ d_1/\sqrt{\bar{z}} \end{pmatrix},\\ 
 \nabla_{\del_z} \Phi &\sim -\frac12 \begin{pmatrix} a_0/z^{3/2} \\ 0 \end{pmatrix},\ &\nabla_{\del_{\bar{z}}} \Phi &\sim -\frac12 \begin{pmatrix} 0 \\ d_1/(\bar{z})^{3/2} \end{pmatrix}.
\end{aligned}
\end{equation*}

Using all of this, we compute that as $\epsilon \to 0$,
\begin{equation*}
T_{1r} - i T_{2r} \sim -\frac{2i}{r}  c_1 \bar{b}_0,\qquad \mbox{and}\qquad 
T_{1r} + i T_{2r} \sim -\frac{2i}{r}  d_1 \bar{a}_0,
\end{equation*}
and hence
\begin{equation}
(\zeta_1 + i\zeta_2)(T_{1r} - i T_{2r}) + (\zeta_1 - i \zeta_2) (T_{1r} + i T_{2r}) = -\frac{2i}{r}\Big(\zeta_1 (c_1 \bar{b}_0 + d_1 \bar{a}_0) + \zeta_2 (d_1 \bar{a}_0 - c_1 \bar{b}_0)\Big).
\label{bdryintegrand}
\end{equation}
After rearranging terms, we can finally rewrite the {\it real} inner product 
\[
\begin{split}
& \mathrm{Re}\, \int_M \langle \calP_0(\zeta), \Phi\rangle = -\frac{1}{2}\mathrm{Re}\int_M\langle \dot{g}_{\zeta},T\rangle= 
\frac{1}{2}\mathrm{Re}\Bigg(\lim_{\varepsilon\rightarrow 0}\int_{\partial M_{\varepsilon}} \mu_r rdyd\theta\Bigg)\\
&=\frac{1}{4}\mathrm{Re}\, \Bigg(\lim_{\varepsilon\rightarrow 0}\int_{\partial M_{\varepsilon}} (\zeta_1 + i\zeta_2)(T_{1r} - i T_{2r}) + (\zeta_1 - i \zeta_2) (T_{1r} + i T_{2r}) rdyd\theta\Bigg)  \\
&=-\frac{1}{2}  \mathrm{Re}\, \int_\Si (i\zeta_1 (c_1 \bar{b}_0 + d_1 \bar{a}_0) + \zeta_2 (d_1 \bar{a}_0 - c_1 \bar{b}_0)  \\
& \qquad  = -\frac14 \int_\Si (  i\zeta_1( c_1 \bar{b}_0 - \bar{c}_1 b_0  + d_1 \bar{a}_0 - \bar{d}_1 a_0)  + \zeta_2 (d_1 \bar{a}_0 + \bar{d}_1 a_0 - c_1 \bar{b}_0 - \bar{c}_1 b_0) \\
& \qquad = -\frac14  \int_\Si  ( \zeta_2 (A + \bar{A}) - i \zeta_1 (A - \bar{A})) = \frac12 \mathrm{Im} \int_\Si \zeta A,
\end{split}
\]
where
\[
\zeta = \zeta_1 + i \zeta_2,\ \ \mbox{and}\ \ A = A(\Phi, \psi_0) = \calT_{\psi_0} \Phi = a_0 \bar{d}_1 - \bar{b}_0 c_1.
\]

We now use all of this to relate the kernel and cokernel of $D\calF$ with those of $(\Dirac, \calD_{\psi_0})$.  

First, suppose that $(\zeta, \phi) \in \ker D\calF$, which can be written as $\calP_0(\zeta) = - \Dirac \phi$.  Note that since 
$\phi \in r H^1_e$, and we are free to assume that $\phi \perp Z_1 = \ker (\Dirac, r H^1_e)$, we have that
$\Dirac \phi$ is orthogonal to $Z_0$, hence the same is true for $\calP_0(\zeta)$. Choosing any $\Phi \in Z_0$, recall 
that $\calB_0(\Phi)$ equals a pair of functions $(a_0, b_0) \in H^{-1/2}(\Si)^2$. The calculations above show that
\[
0 = \mathrm{Re}\, \langle \calP_0(\zeta), \Phi\rangle = \frac12 \mathrm{Im}\, \int_\Si \zeta \calT_{\psi_0}(\Phi) 
= \frac12 \mathrm{Im}\, \int_\Si \zeta (a_0 \bar{d}_1 - \bar{b}_0 c_1).
\]
(This is well-defined since $\zeta \in H^2$.) Rewrite this last expression as
\[
- \frac12 \omega ( (c_1 \zeta, d_1 \bar{\zeta}), (a_0, b_0) ),
\]
where $\omega$ is the symplectic form introduced in Section 3.7.   This vanishes for every pair $(a_0, b_0)$ which arises
as $\calB_0(\Phi)$ for some $\Phi \in Z_0$. However, Corollary 3.28 shows that this set of leading coefficients of elements 
in $Z_0$, which we called the Calderon subspace there, is Lagrangian with respect to $\omega$, which implies
that $(c_1 \zeta, d_1 \bar{\zeta})$ lies in this subspace too.  In other words, there exists some $u \in Z_0$ such
that $\calB_0(u) = (c_1 \zeta, d_1 \bar{\zeta})$.  This $u$ is unique if we also demand that it is orthogonal to $Z_1$. 
Finally, we see immediately that 
\[
\calT_{\phi_0}( c_1 \zeta, d_1 \bar{\zeta}) =  c_1 \zeta \bar{d}_1 - \bar{ d_1} \bar{\zeta} c_1 = 0,
\]
so $u \in \calD_{\psi_0}$. Notice that this implies that $c_1 \zeta$ and $d_1 \bar{\zeta}$ are both smooth, hence $\zeta \in \calC^\infty$. 
In any case, we have now produced a map from $\ker D\calF$ to $\ker (\Dirac, \calD_{\psi_0})$.   This map is
injective, since $\calB_0(u) = (0,0)$ implies $\zeta = 0$, hence $\Dirac \phi = 0$ and thus $\phi = 0$ since $\phi \perp Z_1$. 

We now show that this map is surjective.  Choose any $u \in \ker (\Dirac, \calD_{\psi_0})$.  This is automatically polyhomogeneous
since it is in the nullspace of an elliptic boundary problem, hence the leading coefficients $(a_0, b_0) = \calB_0(u)$ are smooth. 
Now use the boundary condition $\calT_{\psi_0}(a_0, b_0) = 0$ as explained in Section~\ref{Sect_BdryOp} to define 
a smooth function $\zeta = a_0/c_1 = \bar{b}_0/\bar{d}_1$. This is a section of the normal bundle $N\Si$, but can be extended
to be parallel to first order in directions perpendicular to $\Si$.  Now consider $\calP_0(\zeta)$. We compute that for 
any $\Phi \in Z_0$, 
\[
\mathrm{Re}\, \langle \calP_0(\zeta), \Phi \rangle = \frac12 \mathrm{Im}\, \int_\Si \zeta \calT_{\psi_0}(\Phi).
\]
Since the leading coefficients of $u$ are $a_0 = c_1 \zeta$, $b_0 = d_1 \bar{\zeta}$, we can rewrite this further as
\[
\frac12 \omega ( ( c_1 \zeta, d_1 \bar{\zeta}), \calB_0(\Phi) ) = \frac12 \omega ( \calB_0(u), \calB_0(\Phi) ) = 0,
\]
which implies that $\calP_0(\zeta) \in Z_0^\perp$.  There exists, therefore, a unique $\phi \in r H^1_e \cap Z_1^\perp$ such that
$\Dirac \phi = -\calP_0(\zeta)$.  This pair $(\zeta, \phi)$ then lies in $\ker D\calF$, as desired. 

As for the cokernel, suppose that $\Phi \in Z_0$ is orthogonal to the image of $D\calF$. Using the calculations above, 
\[
0 = \mathrm{Re} \langle \calP_0(\zeta) + \Dirac \phi, \Phi \rangle = \frac12 \mathrm{Im} \int_\Si \zeta  \calT_{\psi_0}(\Phi)
\]
for all $\zeta$, since $\langle \Dirac \phi, \Phi \rangle = 0$, which implies that $\calT_{\psi_0}(\Phi) = 0$.  This means that 
$\Phi \in \ker(\Dirac, \calD_{\psi_0})$, and by self-adjointness is also in its cokernel.  We already know that $Z_0^\perp$ equals
the range of $\Dirac$ on $r H^1_e$, hence lies in the range of $D\calF$. Thus the cokernel of $D\calF$ is at least a subspace of the 
cokernel of $(\Dirac, \calD_{\psi_0})$.  Conversely, if $\Phi \perp \mathrm{Ran}\, (\Dirac, \calD_{\psi_0}) = \ker (\Dirac, \calD_{\psi_0})$
then $\calT_{\psi_0}(\Phi) = 0$, so $\calP_0(\zeta) \perp \Phi$ for any $\zeta$, and thus finally, $\Phi \perp \mathrm{Ran}\, D\calF$. 

In summary, we have proved that
\[
\mathrm{ind}\, (D\calF) = \mathrm{ind}\, (\Dirac_0, \calD_{\psi_0}),
\]
where we are interpreting the index as the difference of the dimension of the nullspace and the orthogonal complement of the range. 
This is a slight abuse of language since we do not claim that the range of $D\calF$ is closed. Indeed it is not, and this is the nexus 
of the loss of regularity issue we have mentioned earlier, which makes the analysis of the nonlinear deformation problem so much 
more difficult, see \cites{Takahashi15_Z2HarmSpinors_Arx, Parker-DefHS3M}.    Of course we have
dropped numerous terms throughout this calculation which vanish to some order at $\Si$.  Thus we have actually shown that
$D\calF$ can be modified by subtracting a sequence of compact error terms to arrive at an operator for which the preceding argument
shows has index equal to that of $(\Dirac, \calD_{\psi_0})$.

\bibliography{references}

\end{document}